\documentclass[reqno,a4paper]{amsart}
\usepackage{amssymb}
\usepackage{amsmath}
\begin{document} 
\newtheorem{prop}{Proposition}[section]
\newtheorem{Def}{Definition}[section] \newtheorem{theorem}{Theorem}[section]
\newtheorem{lemma}{Lemma}[section] \newtheorem{Cor}{Corollary}[section]

\title[Maxwell-Klein-Gordon in Lorenz gauge]{\bf Low regularity local well-posedness for the Maxwell-Klein-Gordon equations in Lorenz gauge}
\author[Hartmut Pecher]{
{\bf Hartmut Pecher}\\
Fachbereich Mathematik und Naturwissenschaften\\
Bergische Universit\"at Wuppertal\\
Gau{\ss}str.  20\\
42097 Wuppertal\\
Germany\\
e-mail {\tt pecher@math.uni-wuppertal.de}}
\date{}

\begin{abstract}
The Cauchy problem for the Maxwell-Klein-Gordon equations in Lorenz gauge in two and three space dimensions is locally well-posed for low regularity data without finite energy. The result relies on the null structure for the main bilinear terms which was shown to be not only present in Coulomb gauge but also in Lorenz gauge by Selberg and Tesfahun, who proved global well-posedness for finite energy data in three space dimensions. This null structure is combined with product estimates for wave-Sobolev spaces given systematically by d'Ancona, Foschi and Selberg.
\end{abstract}
\maketitle
\renewcommand{\thefootnote}{\fnsymbol{footnote}}
\footnotetext{\hspace{-1.5em}{\it 2000 Mathematics Subject Classification:} 
35Q61, 35L70 }
\normalsize 
\setcounter{section}{0}
\section{Introduction and main results}
\noindent Consider the Maxwell-Klein-Gordon system 
\begin{align}
\label{1}
\partial^{\nu} F_{\mu \nu} & =  j_{\mu} \\
\label{2}
D^{(A)}_{\mu} D^{(A)\mu} \phi & = m^2 \phi \, ,
\end{align}
where $m>0$ is a constant and
\begin{align}
\label{3}
F_{\mu \nu} & := \partial_{\mu} A_{\nu} - \partial_{\nu} A_{\mu} \\
\label{4}
D^{(A)}_{\mu} \phi & := \partial_{\mu} - iA_{\mu} \phi \\
\label{5}
j_{\mu} & := Im(\phi \overline{D^{(A)}_{\mu} \phi}) = Im(\phi \overline{\partial_{\mu} \phi}) + |\phi|^2 A_{\mu} \, .
\end{align}
Here $F_{\mu \nu} : {\mathbb R}^{n+1} \to {\mathbb R}$ denotes the electromagnetic field, $\phi : {\mathbb R}^{n+1} \to {\mathbb C}$ a scalar field and $A_{\nu} : {\mathbb R}^{n+1} \to {\mathbb R}$ the potential. We use the notation $\partial_{\mu} = \frac{\partial}{\partial x_{\mu}}$, where we write $(x^0,x^1,...,x^n) = (t,x^1,...,x^n)$ and also $\partial_0 = \partial_t$ and $\nabla = (\partial_1,...,\partial_n)$. Roman indices run over $1,...,n$ and greek indices over $0,...,n$ and repeated upper/lower indices are summed. Indices are raised and lowered using the Minkowski metric $diag(-1,1,...,1)$.

The Maxwell-Klein-Gordon system describes the motion of a spin 0 particle with mass $m$ self-interacting with an electromagnetic field.

We are interested in the Cauchy problem with data
$\phi(x,0) = \phi_0(x)$ , $\partial_t \phi(x,0) $ $= \phi_1(x)$ , $F_{\mu \nu}(x,0) =  F^0_{\mu \nu}(x)$ , $A_{\nu}(x,0) = a_{0 \nu}(x)$ , $\partial_t A_{\nu}(x,0) = \dot{a}_{0 \nu}(x)$. The potential $A$ is not uniquely determined but one has gauge freedom. The Maxwell-Klein-Gordon equation is namely invariant under the gauge transformation
$\phi \to \phi' = e^{i\chi}\phi$ , $A_{\mu} \to A'_{\mu} = A_{\mu} + \partial_{\mu} \chi$
for any $\chi: {\mathbb R}^{n+1} \to {\mathbb R}$.

Most of the results obtained so far were given in Coulomb gauge $\partial^j A_j = 0$. Klainerman and Machedon \cite{KM} showed global well-posedness in energy space and above, i.e. for data $\phi_0 \in H^s$ , $\phi_1 \in H^{s-1}$ , $a_{0 \nu} \in H^s$ , $\dot{a} \in H^{s-1}$ with $s \ge 1$ in $n=3$ dimensions improving earlier results of Eardley and Moncrief \cite{EM} for smooth data. They used that the nonlinearities fulfill a null condition in the case of the Coulomb gauge. This global well-posedness result was improved by Keel, Roy and Tao \cite{KRT}, who had only to assume $s > \frac{\sqrt 3}{2}$. Local well-posedness for low regularity data was shown by Cuccagna \cite{C} for $s > 3/4$ and finally almost down to the critical regularity with respect to scaling by Machedon and Sterbenz \cite{MS} for $s > 1/2$, all these results for three space dimensions and in Coulomb gauge. In four space dimensions Selberg \cite{S} showed local well-posedness in Coulomb gauge for $s>1$. Recently Krieger, Sterbenz and Tataru \cite{KST} showed global well-posedness for data with small energy data ($s=1$) for $n=4$, which is the critical space. For space dimension $n \ge 6$ and small critical Sobolev norm for the data local well-posedness was shown by Rodnianski and Tao \cite{RT}. In general the problem seems to be easier in higher dimensions. In temporal gauge local well-posedness was shown for $n=3$ and $ s > 3/4$ for the more general Yang-Mills equations by Tao \cite{T}.

We are interested to consider the Maxwell-Klein-Gordon equations in Lorenz gauge $\partial^{\mu} A_{\mu} = 0$ which was considered much less in the literature because the nonlinear term $Im(\phi \overline{\partial_{\mu} \phi})$ has no null structure. There is a result by Moncrief \cite{M} in two space dimensions for smooth data, i.e. $s \ge 2$. In three space dimensions the most important progress was made by Selberg and Tesfahun \cite{ST} who were able to circumvent the problem of the missing null condition in the equations for $A_{\mu}$ by showing that the decisive nonlinearities in the equations for $\phi$ as well as $F_{\mu \nu}$ fulfill such a null condition which allows to show that global well-posedness holds for finite energy data, i.e. $\phi_0 \in H^1$, $\phi_1 \in L^2$ , $F^0_{\mu \nu} \in L^2$ , $a_{0\nu} \in \dot{H}^1$ , $\dot{a}_{0 \nu} \in L^2$, and three space dimensions, where $\phi \in C^0({\mathbb R},H^1) \cap C^1({\mathbb R},L^2)$ and $F_{\mu \nu} \in C^0({\mathbb R},L^2)$. The potential possibly loses some regularity compared to the data but as remarked also by the authors this is not the main point because one is primarily interested in the regularity of $\phi$ and $F_{\mu \nu}$. Persistence of higher regularity for the solution also holds.

A null structure in Lorenz gauge was first detected for the Maxwell-Dirac system by d'Ancona, Foschi and Selberg \cite{AFS1}.

The paper \cite{ST} is the basis for our results. We show that local well-posedness can also be proven for less regular data without finite energy, namely for $s > 3/4$ in space dimension $n=3$ and also in dimension $n=2$ for a slightly different data space. We show that in the case $n=2$ corresponding null conditions also hold. When this has been done the necessary estimates in Bourgain type spaces mainly rely on the bilinear estimates in wave-Sobolev spaces given by d'Ancona, Foschi and Selberg for $n=3$ \cite{AFS3} and $n=2$ \cite{AFS2}.

We now formulate our main result. We assume the Lorenz condition
\begin{equation}
\label{6}
\partial^{\mu} A_{\mu} = 0
\end{equation}
and Cauchy data
\begin{align}
\label{7}
\phi(x,0) &= \phi_0(x) \in H^s \quad , \quad \partial_t \phi(x,0) = \phi_1(x) \in H^{s-1} \, , \\
\label{8}
F_{\mu \nu}(x,0)& = F^0_{\mu \nu}(x) \,\mbox{with}\, F^0_{\mu \nu} \in H^{s-1} \, \mbox{for} \, n=3 \,\mbox{and}\, |\nabla|^{-\epsilon} F^0_{\mu \nu} \in H^{s-1+\epsilon} \,\mbox{for}\, n=2 \, , 
\end{align}
where $\epsilon$ is a small positive constant and
\begin{equation}
\label{9}
A_{\nu}(x,0) = a_{0 \nu}(x) \quad , \quad \partial_t A_{\nu}(x,0) = \dot{a}_{0 \nu}(x) \, ,
\end{equation} 
which fulfill the following conditions
\begin{equation}
\label{10}
a_{00} = \dot{a}_{00} = 0 \, , 
\end{equation}
\begin{equation}
\label{11}
\nabla a_{0j} \in H^{s-1} \, , \,  \dot{a}_{0j} \in H^{s-1}\, \mbox{for} \, n=3 \, ,
\end{equation}
\begin{equation}
\label{11'}
|\nabla|^{1-\epsilon} a_{0j} \in H^{s-1+\epsilon} \, , \, |\nabla|^{-\epsilon}\dot{a}_{0j} \in H^{s-1+\epsilon} \, \mbox{for} \, n=2 \, ,
\end{equation}
\begin{equation}
\label{12}
\partial^k a_{0k} = 0 \, ,
\end{equation}
\begin{equation}
\label{13}
\partial_j a_{0k} - \partial_k a_{0j} = F^0_{jk} \, ,
\end{equation}
\begin{equation}
\label{14}
\dot{a}_{0k} = F^0_{0k} \, , 
\end{equation}
\begin{equation}
\label{15}
\partial^k F^0_{0k}  = Im(\phi_0 \overline{\phi}_1) \, .
\end{equation}
(\ref{10}) can be assumed because otherwise the Lorenz condition does not determine the potential uniquely. (\ref{12}) follow from the Lorenz condition (\ref{6}) in connection with (\ref{10}). (\ref{13}) follows from (\ref{3}), similarly (\ref{14}) from (\ref{3}) and (\ref{10}). (\ref{1}) requires
$$ \partial^k F^0_{0k} = j_0(0) = Im(\phi_0 \overline{\phi}_1) + |\phi_0|^2 a_{00} = Im(\phi_0 \overline{\phi}_1) \, $$
thus (\ref{15}). By (\ref{12}) we have
$$ \Delta a_{0j} = \partial^k \partial_k a_{0j} = \partial^k(\partial^j a_{0k} - F^0_{jk}) = - \partial^k F^0_{jk} \, , $$
so that $a_{0j}$ is uniquely determined as
$$ a_{0j} = (-\Delta)^{-1} \partial^k F^0_{jk} $$
and fulfills (\ref{11}) and (\ref{11'}).

We define the wave-Sobolev spaces $X^{s,b}_{\pm}$ as the completion of the Schwarz space $\mathcal{S}({\mathbb R}^{n+1})$ with respect to the norm
$$ \|u\|_{X^{s,b}_{\pm}} = \| \langle \xi \rangle^s \langle  \tau \pm |\xi| \rangle^b \widehat{u}(\tau,\xi) \|_{L^2_{\tau \xi}}  $$
and $X^{s,b}_{\pm}[0,T]$ as the space of the restrictions to $[0,T] \times \mathbb{R}^n$.

Our main theorem reads as follows:
\begin{theorem}
\label{Theorem1}
Assume $n=3$ and $s = \frac{3}{4} + \delta$ , $r=\frac{1}{2}+\delta$ or $n=2$ and $s=\frac{3}{4}+\delta$, $r=\frac{1}{4}+\delta$ 
, where $\delta > 0$ is a small number. The data are assumed to fulfill (\ref{7}) - (\ref{15}). Then the öproblem (\ref{1}) - (\ref{6}) has a unique local solution
$$ \phi \in X_+^{s,\frac{1}{2}+}[0,T] +X_-^{s,\frac{1}{2}+}[0,T] \, , \, \partial_t \phi \in X_+^{s-1,\frac{1}{2}+}[0,T] +X_-^{s-1,\frac{1}{2}+}[0,T]$$
and
$$ F_{\mu \nu} \in X_+^{s-1,\frac{1}{2}+}[0,T] +X_-^{s-1,\frac{1}{2}+}[0,T]$$
in the case $n=3$ and
$$ |\nabla|^{-\epsilon} F_{\mu \nu} \in X_+^{s-1+\epsilon,\frac{1}{2}+}[0,T] +X_-^{s-1+\epsilon,\frac{1}{2}+}[0,T]$$
in the case $n=2$ relative to a potential $A=(A_0,A_1,...,A_n)$, where
$A = A^{hom}_+ + A^{hom}_- + A^{inh}_+ + A^{inh}_-$ with
$|\nabla| A^{hom}_{\pm} \in X^{r-1,1-\epsilon_0}_{\pm}[0,T]$ and $A^{inh}_{\pm} \in X^{r,1-\epsilon_0}_{\pm}[0,T]$ for $n=3$ and $|\nabla|^{1-\epsilon} A^{hom}_{\pm} \in X^{r-1+\epsilon,1-\epsilon_0}_{\pm}[0,T]$ and $|\nabla|^{\epsilon_1} A^{inh}_{\pm} \in X^{r-\epsilon_1,1-\epsilon}_{\pm}[0,T]$ for $n=2$, where $\epsilon$ , $0 < \epsilon _0 < \delta$ and $\epsilon_1$ are small positive numbers. 
\end{theorem}
\noindent{\bf Remarks:} 
\begin{itemize}
\item We immediately obtain $\phi \in C^0([0,T],H^s)$ and $F_{\mu \nu} \in C^0([0,T],H^{s-1})$ for n=3 and $|\nabla|^{-\epsilon} F_{\mu \nu} \in C^0([0,T],H^{s-1+\epsilon})$ for n=2. 
\item As usual for solutions constructed by Banach's fixed point theorem in $X^{s,b}$-type spaces we also have persistence of higher regularity, i.e. $\delta > 0$ can be chosen arbitrarily in Theorem \ref{Theorem1}.
\end{itemize}

We can reformulate the system (\ref{1}),(\ref{2}) under the Lorenz condition (\ref{6}) as follows:
$$\square A_{\mu} = \partial^{\nu} \partial_{\nu} A_{\mu} = \partial^{\nu}(\partial^{\mu} A_{\nu} - F_{\mu \nu}) = -\partial^{\nu} F_{\mu \nu} = - j_{\mu} \, , $$
thus (using the notation $\partial = (\partial_0,\partial_1,...,\partial_n)$):
\begin{equation}
\label{16}
\square A = -Im (\phi \overline{\partial \phi}) - A|\phi|^2 =: N(A,\phi) 
\end{equation}
and
\begin{align*}
m^2 \phi & = D^{(A)}_{\mu} D^{(A)\mu} \phi = \partial_{\mu} \partial^{\mu} \phi -iA_{\mu} \partial^{\mu} \phi -i\partial_{\mu}(A^{\mu} \phi) - A_{\mu}A^{\mu} \phi \\
& = \square \phi - 2i A^{\mu} \partial_{\mu} \phi - A_{\mu} A^{\mu} \phi \,
\end{align*}
thus
\begin{equation}
\label{17}
(\square -m^2) \phi = 2i A^{\mu} \partial_{\mu} \phi + A_{\mu} A^{\mu} \phi =: M(A,\phi) \, .
\end{equation}
Conversely, if $\square A_{\mu} = -j_{\mu}$ and $F_{\mu \nu} := \partial_{\mu} A_{\nu} - \partial_{\nu} A_{\mu}$ and the Lorenz condition (\ref{6}) holds then
$$ \partial^{\nu} F_{\mu \nu} = \partial^{\nu}(\partial_{\mu} A_{\nu} - \partial_{\nu} A_{\mu}) = \partial_{\mu} \partial^{\nu} A_{\nu} - \partial^{\nu} \partial_{\nu} A_{\mu} = -\square A_{\mu} = j_{\mu} \, $$
thus (\ref{1}),(\ref{2}) is equivalent to (\ref{16}),(\ref{17}), if (\ref{3}),(\ref{4}) and (\ref{6}) are satisfied.

The paper is organized as follows: in chapter 2 we prove the null structure of $A^{\mu} \partial_{\mu} \phi$ and in the Maxwell part. In chapter 3 local well-posedness of (\ref{16}),(\ref{17}) is shown with regularity of $\phi$ and $A$ as specified in Theorem \ref{Theorem1}. This relies on the null structure of $A^{\mu} \partial_{\mu} \phi$ and the bilinear estimates in wave-Sobolev spaces by d'Ancona, Foschi and Selberg (\cite{AFS3} and \cite{AFS2}). In chapter 4 we show that the Maxwell-Klein-Gordon system is satisfied and $F_{\mu \nu}$ has the desired regularity properties using the null structure of the Maxwell part and again the bilinear estimates by d'Ancona, Foschi and Selberg.

\section{Null structure}
\noindent{\bf Null structure of $A_{\mu} \partial_{\mu}\phi$.} \\
Using the Riesz transform $R_k := (-\Delta)^{-\frac{1}{2}} \partial_k $ we have
$$ A^j = R^k(R_j A_k - R_k A_j) - R_j R_k A^k \, , $$
so that
\begin{align*}
A^{\mu} \partial_{\mu} \phi & = -A^0 \partial_0 \phi + A^j \partial_j \phi \\
& = (-A_0 \partial_0 \phi -(R_j R_k A^k)\partial^j \phi) + (R^k(R_j A_k - R_k A_j)\partial^j \phi) \\
& =: P_1 + P_2
\end{align*}
Now define
$$ Q_{jk}(\phi,\psi) := \partial_j \phi \partial_k \psi - \partial_k \phi \partial_j \psi \, , $$
so that
\begin{align*}
&\sum_{j,k} Q_{jk}(\phi,(-\Delta)^{-\frac{1}{2}}(R_j A_k - R_k A_j)) \\
& = \sum_{j,k} [\partial_j \phi \partial_k((-\Delta)^{-\frac{1}{2}}(R_j A_k - R_k A_j)) - \partial_k \phi \partial_j((-\Delta)^{-\frac{1}{2}}(R_j A_k - R_k A_j))] \\
& = \sum_{j,k} [\partial_j \phi (-\Delta)^{-1}(\partial_k \partial_j A_k - \partial_k^2 A_j) - \partial_k \phi(-\Delta)^{-1}(\partial_j^2 A_k - \partial_j \partial_k A_j)] \\
&=2(\sum_{j,k} \partial_j \phi (-\Delta)^{-1} \partial_k \partial_j A_k + \sum_j \partial_j \phi A_j) \\
& = 2(\sum_{j,k} \partial_j \phi R_k R_j A_k - \sum_{j,k} \partial_j \phi R^k R_k A_j)\\
& = -2P_2 \, .
\end{align*}
We have
\begin{align*}
& Q_{jk}(\phi,(-\Delta)^{-\frac{1}{2}}(R_j A_k - R_k A_j)) \\
& = (-\Delta)^{-1} [\partial_k (\partial_j A_k - \partial_k A_j)\partial_j \phi - \partial_j(\partial_j A_k - \partial_k A_j)\partial_k \phi] \, ,
\end{align*}
so that the symbol of $P_2$ is bounded by
\begin{equation}
\label{2.0} \sum_{j,k} \left|\frac{1}{|\eta|^2}(\eta_k \zeta_j - \eta_j \zeta_k) (\eta_j - \eta_k) \right| \lesssim \sum_{j,k} \frac{|\eta_k \zeta_j - \eta_j \zeta_k|}{|\eta|} \lesssim |\zeta| \angle(\eta,\zeta) \, ,
\end{equation}
where $\angle(\eta,\zeta)$ denotes the angle between $\eta$ and $\zeta$.

Before we consider $P_1$ we define
$$ A_{\pm} := \frac{1}{2}(A \pm (i|\nabla|)^{-1} A_t) \, , $$
so that $A=A_+ + A_-$ and $A_t = i|\nabla|(A_+-A_-)$, and
$$\phi_{\pm} := \frac{1}{2}(\phi \pm (i \langle \nabla \rangle_m)^{-1} \phi_t) $$
with $\langle \nabla \rangle_m := (m^2-\Delta)^{\frac{1}{2}}$, so that $\phi = \phi_+ + \phi_-$ and $\phi_t = i \langle \nabla \rangle_m (\phi_+ - \phi_-)$. 

We transform (\ref{16}),(\ref{17}) into
\begin{align}
\label{2.1}
(i\partial_t \pm \langle \nabla \rangle_m) \phi_{\pm} & = -(\pm 2 \langle \nabla \rangle_m)^{-1} M(A,\phi) \\
\label{2.2}
(i\partial_t \pm |\nabla|) A_{\pm} & = -(\pm 2 |\nabla|)^{-1} N(A,\phi) \, ,
\end{align}
so that under the Lorenz condition $\partial_k A^k = \partial_t A^0$ we have
\begin{align*}
P_1 &= -A^0 \partial_0 \phi -(R_j R_k A^k)\partial^j \phi \\
& = -A^0 \partial_t \phi- (-\Delta)^{-1} \partial_j \partial_k A^k)\partial^j \phi \\
& = -A^0 \partial_t \phi - (-\Delta)^{-1} \nabla \partial_t A^0 \cdot \nabla \psi \, ,
\end{align*}
which implies
\begin{align}
\label{2.2a}
iP_1 & = (A_{0+}+A_{0-})\langle \nabla \rangle_m (\phi_+ - \phi_-) + |\nabla|^{-1} \nabla(A_{0+} - A_{0-}) \cdot \nabla(\phi_+ + \phi_-) \\
\nonumber
& = \sum_{\pm_1,\pm_2} \pm_2 \mathcal{A}_{(\pm_1,\pm_2)} (A_{0 \pm1},\phi_{\pm 2}) \, ,
\end{align}
where
$$ \mathcal{A}_{(\pm_1,\pm_2)}(f,g) := f \langle \nabla \rangle_m g + |\nabla|^{-1} \nabla(\pm_1 f) \cdot \nabla(\pm_2 g) \, . $$
Its symbol $a_{(\pm_1,\pm_2)} (\eta,\zeta)$ is bounded by the elementary estimate (\cite{ST}, Lemma 3.1):
\begin{equation}
\label{2.3}
|a_{(\pm_1,\pm_2)} (\eta,\zeta) | \sim \left|\langle \zeta \rangle_m - \frac{(\pm_1 \eta) \cdot (\pm_2 \zeta)}{|\eta|}\right| \lesssim m + |\zeta| \angle(\pm_1 \eta, \pm_2 \zeta) \,.
\end{equation}
 
We have shown that the quadratic term in the nonlinearity $M(A,\phi)$ in the  equation (\ref{17}) for $\phi$ fulfills a null condition. \\[0.5em]
{\bf Null structure in the Maxwell part.}\\
We start from Maxwell's equations (\ref{1}), i.e. $-\partial^0 F_{l0} + \partial^k F_{lk} = j_l$ and $\partial^k F_{0k} = j_0$ and obtain
\begin{align}
\nonumber
\square F_{k0} & = -\partial_0(\partial_0 F_{k0}) + \partial^l \partial_l F_{k0} \\ 
\nonumber
& = -\partial_0(\partial^l F_{kl} - j_k) + \partial^l \partial_l F_{k0} \\
\nonumber
& = -\partial^l \partial_0(\partial_k A_l - \partial_l A_k) + \partial_0 j_k + \partial^l \partial_l F_{k0} \\
\nonumber
& = -\partial^l[\partial_k(\partial_0 A_l - \partial_l A_0) - \partial_l(\partial_0 A_k - \partial_k A_0)] + \partial_0 j_k + \partial^l \partial_l F_{k0} \\
\nonumber
& = -\partial^l \partial_k F_{0l} + \partial^l \partial_l F_{0k} + \partial_0 j_k + \partial^l \partial_l F_{k0} \\
\nonumber
& =  -\partial^l \partial_k F_{0l} +\partial_0 j_k \\
\label{2.10}
& = -\partial_k j_0 + \partial_0 j_k
\end{align}
and
\begin{align}
\nonumber
\square F_{kl} & = -\partial_0 \partial_0 F_{kl} + \partial^m \partial_m F_{kl} \\\nonumber
& = -\partial_0 \partial_0 (\partial_k A_l - \partial_l A_k) + \partial^m \partial_m F_{kl} \\\nonumber
& = -\partial_0 \partial_k(\partial_0 A_l - \partial_l A_0) + \partial_0 \partial_l (\partial_0 A_k - \partial_k A_0) + \partial^m \partial_m F_{kl} \\\nonumber
& = -\partial_0 \partial_k F_{0l} + \partial_0 \partial_l F_{0k} + \partial^m \partial_m F_{kl} \\\nonumber
& = \partial_k \partial_0 F_{l0} - \partial_l \partial_0 F_{k0} + \partial^m \partial_m F_{kl} \\\nonumber
& = \partial_k (\partial^m F_{lm} - j_l) - \partial_l(\partial^m F_{km} - j_k)+ \partial^m \partial_m F_{kl} \\\nonumber
& = \partial_k \partial^m F_{lm} - \partial_l \partial^m F_{km} + \partial^m \partial_m F_{kl} +\partial_l j_k - \partial_k j_l \\\nonumber
& = \partial_k \partial^m(\partial_l A_m - \partial_m A_l) - \partial_l \partial^m(\partial_k A_m - \partial_m A_k)+ \partial^m \partial_m F_{kl} +\partial_l j_k - \partial_k j_l \\\nonumber
& = \partial^m \partial_m (\partial_l A_k - \partial_k A_l) + \partial^m \partial_m F_{kl} +\partial_l j_k - \partial_k j_l \\\nonumber
& = \partial^m \partial_m F_{lk} + \partial^m \partial_m F_{kl} +\partial_l j_k - \partial_k j_l \\
\label{2.11}
&   = \partial_l j_k - \partial_k j_l \, .
\end{align}
By the definition ({5}) of $j_{\mu}$ we obtain
\begin{align}
\nonumber
\partial_0 j_k - \partial_k j_0 
 =& \, Im(\partial_0 \phi \overline{\partial_k \phi}) + Im(\phi \overline{\partial_0 \partial_k \phi}) + \partial_0(A_k |\phi|^2) \\
 \nonumber
&- Im(\partial_k \phi \overline{\partial_0 \phi}) - Im(\phi \overline{\partial_k \partial_0 \phi}) - \partial_k(A_0 |\phi|^2) \\
\label{2.10a}
 =&\, Im(\partial_t \phi\overline{\partial_k \phi} - \partial_k \phi \overline{\partial_t \phi}) + \partial_t(A_k |\phi|^2) - \partial_k(A_0 |\phi|^2)
\end{align}
and
\begin{align}
\nonumber
\partial_l j_k - \partial_k j_l 
 =& \, Im(\partial_l \phi \overline{\partial_k \phi}) + Im(\phi \overline{\partial_l \partial_k \phi}) + \partial_l(A_k |\phi|^2) \\
 \nonumber
&- Im(\partial_k \phi \overline{\partial_l \phi}) - Im(\phi \overline{\partial_k \partial_l \phi}) - \partial_k(A_l |\phi|^2) \\
\label{2.11b}
 =&\, Im(\partial_l \phi\overline{\partial_k \phi} - \partial_k \phi \overline{\partial_l \phi}) + \partial_l(A_k |\phi|^2) - \partial_k(A_l |\phi|^2) \, .
\end{align}
After the decomposition $\phi = \phi_+ + \phi_-$ we have
\begin{equation}
\label{2.11c}
\partial_l \phi\overline{\partial_k \phi} - \partial_k \phi \overline{\partial_l \phi} = \sum_{\pm_1,\pm_2} \mathcal{C}_{\pm_1,\pm_2} (\phi_{\pm_1,\phi_{\pm_2}})                            \, , \end{equation}
where
\begin{equation}
\label{2.12}
\mathcal{C}_{\pm_1,\pm_2} (f,g) := \partial_l f \overline{\partial_k g} - \partial_k f  \overline{\partial_l g} \, .
\end{equation}
Its symbol
\begin{equation}
c_{\pm_1,\pm_2}(\eta,\zeta) = \eta_l \zeta_k - \eta_k \zeta_l
\label{2.13}
\end{equation}
fulfills
\begin{equation}
\label{2.14}
|c_{\pm_1,\pm_2} (\eta,\zeta) = |(\pm_1 \eta_l)(\pm_2 \zeta_k) - (\pm_1 \eta_k)(\pm_2 \zeta_l)| \lesssim |\eta||\zeta| \angle(\pm_1 \eta,\pm_2 \zeta) \, . 
\end{equation}
Similarly using $\partial_t \phi = i \langle \nabla \rangle_m (\phi_+ - \phi_-)$ we have
\begin{equation}
\label{2.10b}
\partial_t \phi\overline{\nabla \phi} - \nabla \phi \overline{\partial_t \phi} = \sum_{\pm_1,\pm_2} (\pm_1 1)(\pm_2 1)\mathcal{B}_{\pm_1,\pm_2} (\phi_{\pm_1,\phi_{\pm_2}})                            \, , \end{equation}
where
\begin{equation}
\label{2.15}
\mathcal{B}_{\pm_1,\pm_2} (f,g) := i(\langle \nabla \rangle_m f \overline{\nabla (\pm_2 g)} - \nabla (\pm_1 f)  \overline{\langle \nabla \rangle_m g}) \, .
\end{equation}
Its symbol
\begin{equation}
b_{\pm_1,\pm_2}(\eta,\zeta) = \langle \eta \rangle_m (\pm_2 \zeta) - \langle\zeta \rangle_m (\pm_1 \eta)
\label{2.16}
\end{equation}
can be estimated elementarily (\cite{ST}, Lemma 3.2):
\begin{equation}
\label{2.17}
|b_{\pm_1,\pm_2} (\eta,\zeta)| \lesssim m(|\eta| + |\zeta|) + |\eta||\zeta| \angle(\pm_1 \eta,\pm_2 \zeta) \, . 
\end{equation}
We have shown that the quadratic terms in the equations (\ref{2.10}) and (\ref{2.11}) for $F_{k0}$ and $F_{kl}$ fulfill a null condition.

\section{Local well-posedness}
Recall $\phi_{\pm}=\frac{1}{2}(\phi \pm (i\langle \nabla \rangle_m)^{-1}\phi_t)$ , so that $\phi=\phi_+ + \phi_-$ and $ \partial_t \phi = i \langle \nabla \rangle_m(\phi_+ - \phi_-)$, and $A_{\pm} = \frac{1}{2}(A \pm (i|\nabla|)^{-1} A_t)$ so that $A=A_+ + A_-$ and $\partial_t A = i|\nabla|(A_+-A_-)$, we write (\ref{2.1}),(\ref{2.2}) as follows:
\begin{align}
\label{3.1}
(i\partial_t \pm \langle \nabla \rangle_m) \phi_{\pm} & = - (\pm 2 \langle \nabla \rangle_m)^{-1} \mathcal{M}(\phi_+,\phi_-,A_+,A_-) \\
\label{3.2}
(i\partial_t \pm |\nabla|)A_{\pm} & = -(\pm 2|\nabla|)^{-1} \mathcal{N}(\phi_+,\phi_-,A_+,A_-) \, ,
\end{align}
where
\begin{align}
\label{3.3}
\mathcal{M}(\phi_+,\phi_-,A_+,A_-) & =2 \sum_{\pm_1,\pm_2} \pm_2 \mathcal{A}_{(\pm_1,\pm_2)} (A_{0\pm 1},\phi_{\pm 2}) + 2iP_2 + A_{\mu}A^{\mu} \phi \\
\label{3.4}
\mathcal{N}_0(\phi_+,\phi_-,A_+,A_-) & = Im(\phi i \langle \nabla \rangle_m(\overline{\phi}_+ - \overline{\phi}_-)) - A_0 |\phi|^2 \\
\label{3.5}
\mathcal{N}_j(\phi_+,\phi_-,A_+,A_-) & = -Im(\phi \overline{\partial_j \phi}) -A_j |\phi|^2 \, .
\end{align}
The initial data are
\begin{align}
\label{3.5a}
\phi_{\pm}(0) & = \frac{1}{2}(\phi_0 \pm (i \langle \nabla \rangle_m)^{-1} \phi_1) \\
\label{3.5b}
A_{0\pm}(0)& = \frac{1}{2}(a_{00}\pm(i|\nabla|^{-1})\dot{a}_{00}) = 0 \\
\label{3.5c}
A_{j\pm}(0) & = \frac{1}{2}(a_{0j} \pm (i|\nabla|)^{-1} \dot{a}_{0j}) \, .
\end{align}
(\ref{3.5b}) follows from (\ref{10}). From (\ref{7}) we have $\phi_{\pm}(0) \in H^s$, and from (\ref{11}) we have for $r\le s$ in the case $n=3$: $\nabla a_{0j} \in H^{r-1}$ , $\dot{a}_{0j} \in H^{r-1}$, so that $\nabla A_{j\pm}(0) \in H^{r-1}$, whereas in the case $n=2$ we have $|\nabla|^{1-\epsilon} a_{0j} \in H^{r-1+\epsilon}$ , $|\nabla|^{-\epsilon} \dot{a}_{0j} \in H^{r-1+\epsilon}$, so that $|\nabla|^{1-\epsilon} A_{j\pm}(0) \in H^{r-1+\epsilon}$.

We split $A_{\pm} = A_{\pm}^{hom} + A_{\pm}^{inh}$ into its homogeneous and inhomogeneous part, where $(i\partial_t \pm|\nabla|)A_{\pm}^{hom} =0$ with data as in (\ref{3.5b}) and (\ref{3.5c}) and $A_{\pm}^{inh}$ is the solution of (\ref{3.2}) with zero data. By the linear theory we obtain for $b>1/2$:
$$ \|\phi_{\pm}^{hom}\|_{X^{s,b}_{\pm}} \lesssim \|\phi_{\pm}(0)\|_{H^s}$$
and for $\beta > 1/2$:
$$ \| |\nabla| A_{\pm}^{hom} \|_{X_{\pm}^{r-1,\beta}} \lesssim \| |\nabla| A_{\pm}(0)\|_{H^{r-1}} \quad \mbox{for} \, n=3 \, $$
$$\| |\nabla|^{1-\epsilon} A_{\pm}^{hom} \|_{X_{\pm}^{r-1+\epsilon,\beta}} \lesssim \| |\nabla|^{1-\epsilon} A_{\pm}(0)\|_{H^{r-1+\epsilon}} \quad \mbox{for} \, n=2 \, . $$

Our aim is to show the following local well-posedness result:
\begin{theorem}
\label{Theorem2}
Assume
\begin{equation}
\label{3.9}
s= \frac{3}{4} + \delta \, , \, r = \frac{1}{2} + \delta \,\,  \mbox{for}\, \, n=3 \,\, \mbox{and}\, \, s= \frac{3}{4} + \delta \, , \, r = \frac{1}{4} + \delta \, \, \mbox{for} \,\, n=2 \, ,
\end{equation}
where $\delta >0$ is a small number. The system (\ref{3.1}),(\ref{3.2}) for data $\phi_{\pm}(0) \in H^s$ and $|\nabla| A_{\pm}(0) \in H^{r-1}$ for $n=3$ and $|\nabla|^{1-\epsilon} A_{\pm}(0) \in H^{r-1+\epsilon}$ for $n=2$ has a unique local solution $\phi_{\pm} \in X_{\pm}^{s,\frac{1}{2}+}[0,T]$ and $|\nabla| A_{\pm}^{hom} \in X_{\pm}^{r-1,1-\epsilon_0}[0,T]$ , $A_{\pm}^{inh} \in X_{\pm}^{r,1-\epsilon_0}[0,T]$ for $n=3$, and $|\nabla|^{1-\epsilon} A_{\pm}^{hom} \in X_{\pm}^{r-1+\epsilon,1-\epsilon_0}[0,T]$ , $|\nabla|^{\epsilon_1} A_{\pm}^{inh}  \in $ \\ $X_{\pm}^{r-\epsilon_1,1-\epsilon_0}[0,T]$ for $n=2$, where $\epsilon,\epsilon_0,\epsilon_1 > 0$ are small enough.
\end{theorem}

Fundamental for its proof are the following bilinear estimates in wave-Sobolev spaces which were proven by d'Ancona, Foschi and Selberg in the cases $n=2$ in \cite{AFS2} and $n=3$ in \cite{AFS3} in a more general form which include many limit cases which we do not need.
\begin{theorem}
\label{Theorem3}
Let $n=2$ or $n=3$. The estimate
$$\|uv\|_{X_{\pm}^{-s_0,-b_0}} \lesssim \|u\|_{X^{s_1,b_1}_{\pm_1}} \|v\|_{X^{s_2,b_2}_{\pm_2}} $$
holds, provided the following conditions hold:
\begin{align}
\nonumber
& b_0,b_1,b_2 \ge 0 \\
\nonumber
& b_0 + b_1 + b_2 > \frac{1}{2} \\
\nonumber
& b_0 + b_1 > 0 \\
\nonumber
& b_0 + b_2 > 0 \\
\nonumber
& b_1 + b_2 > 0 \\
\nonumber
&s_0+s_1+s_2 > \frac{n+1}{2} -(b_0+b_1+b_2) \\
\nonumber
&s_0+s_1+s_2 > \frac{n}{2} -(b_0+b_1) \\
\nonumber
&s_0+s_1+s_2 > \frac{n}{2} -(b_0+b_2) \\
\nonumber
 &s_0+s_1+s_2 > \frac{n}{2} -(b_1+b_2)\\
\nonumber
&s_0+s_1+s_2 > \frac{n-1}{2} - b_0 \\
\nonumber
&s_0+s_1+s_2 > \frac{n-1}{2} - b_1 \\
\nonumber
&s_0+s_1+s_2 > \frac{n-1}{2} - b_2 \\
\nonumber
&s_0+s_1+s_2 > \frac{n+1}{4} \\
 \label{5.13}
&(s_0 + b_0) +2s_1 + 2s_2 > \frac{n}{2} \\
\nonumber
&2s_0+(s_1+b_1)+2s_2 > \frac{n}{2} \\
\nonumber
&2s_0+2s_1+(s_2+b_2) > \frac{n}{2} \\
\nonumber
&s_1 + s_2 > 0 \\
\nonumber
&s_0 + s_2 > 0 \\
\nonumber
&s_0 + s_1 > 0 \, .
\end{align}
If $n=3$ the condition (\ref{5.13}) is only necessary in the case when $\langle \xi_0 \rangle \lesssim \langle \xi_1 \rangle \sim \langle \xi_2 \rangle$ and also $\pm_1$ and $\pm_2$ are different signs. Here $\xi_0$,$\xi_1$ and $\xi_2$ denote the spatial frequencies of $\widehat{uv}$,$\widehat{u}$ and $\widehat{v}$, respectively.
\end{theorem}

The second basic fact used in the sequel is the following elementary estimate for the angle between two vectors.
\begin{lemma}
\label{Lemma} Assume $0 \le \alpha \le \frac{1}{2}$.
For arbitrary signs $\pm,\pm_1,\pm_2$ the following estimate holds:
\begin{equation}
\label{angle1}
\angle (\pm_1 \eta,\pm_2(\xi - \eta))  \lesssim \frac{\langle \tau \pm |\xi| \rangle^{\frac{1}{2}-\alpha}}{\min(|\eta|, |\eta - \xi|)^{\frac{1}{2}-\alpha}}   + \frac{ \langle \lambda \pm_1 |\eta| \rangle^{\frac{1}{2}} 
+ \langle \tau - \lambda \pm_2 |\xi - \eta |\rangle^{\frac{1}{2}}}{\min(|\eta|,|\eta-\xi|)^{\frac{1}{2}}}
\end{equation}
In the case of different signs $\pm_1$ and $\pm_2$ the following (improved) estimate holds:
\begin{align}
\nonumber
\angle (\pm_1 \eta,\pm_2(\xi - \eta)) & \lesssim \frac{|\xi|^{\frac{1}{2}}}{|\eta|^{\frac{1}{2}} |\eta - \xi|^{\frac{1}{2}}} \big( \min(|\eta|,|\eta - \xi|)^{\alpha} \langle \tau \pm |\xi| \rangle^{\frac{1}{2}-\alpha} + \langle \lambda \pm_1 |\eta| \rangle^{\frac{1}{2}} \\
\nonumber
& \hspace{5em} + \langle \tau - \lambda \pm_2 |\xi - \eta |\rangle^{\frac{1}{2}}\big) \\
\label{angle2}
& \lesssim \frac{|\xi|^{\frac{1}{2}}}{|\eta|^{\frac{1}{2}} |\eta - \xi|^{\frac{1}{2}}} \min(|\eta|,|\eta - \xi|)^{\alpha} \langle \tau \pm |\xi| \rangle^{\frac{1}{2}-\alpha} \\ \nonumber
&
\hspace{5em} + \frac{\langle \lambda \pm_1 |\eta| \rangle^{\frac{1}{2}} + \langle \tau - \lambda \pm_2 |\xi - \eta \rangle^{\frac{1}{2}}}{\min(|\eta|,|\eta - \xi|)^{\frac{1}{2}}} \, .
\end{align}
\end{lemma}
\begin{proof} These results follow from \cite{AFS}, Lemma 7 and the considerations ahead of and after that lemma, where we use that in the case of different signs $\pm_1$ and $\pm_2$ we have
$\angle (\pm_1 \eta,\pm_2 (\xi - \eta)) =  \angle (\eta,\eta - \xi)$. Cf. also \cite{ST}, Lemma 4.3.
\end{proof}
\begin{proof} [Proof of Theorem \ref{Theorem2}]
It is standard that the claimed result follows by the contraction mapping principle in connection with the linear theory, if the following estimates hold:
\begin{align}
\label{3.6}
\| \langle \nabla \rangle_m^{-1} \mathcal{M}(\phi_+,\phi_-,A_+,A_-)\|_{X^{s,-\frac{1}{2}++}_{\pm}} & \lesssim R^2 + R^3 
\end{align}
and
\begin{align}
\label{3.7}
\| | \nabla |^{-1} \mathcal{N}(\phi_+,\phi_-,A_+,A_-)\|_{X^{r,-\epsilon_0 +}_{\pm}} & \lesssim R^2 + R^3 \quad \mbox{for} \, n=3 \\
\label{3.8}
\| | \nabla |^{-1+\epsilon_1} \mathcal{N}(\phi_+,\phi_-,A_+,A_-)\|_{X^{r-\epsilon_1,-\epsilon_0 +}_{\pm}} & \lesssim R^2 + R^3 \quad \mbox{for} \, n=2 \, ,
\end{align}
where
$$
R  = \sum_{\pm} (\|\phi_{\pm}\|_{X^{s,\frac{1}{2}+}_{\pm}} + \| |\nabla| A_{\pm}\|_{X^{r-1,1-\epsilon_0}_{\pm}}) \quad \mbox{for} \, n=3 \, , $$
$$
R  = \sum_{\pm} (\|\phi_{\pm}\|_{X^{s,\frac{1}{2}+}_{\pm}} + \| |\nabla|^{1-\epsilon} A_{\pm}\|_{X^{r-1+\epsilon,1-\epsilon_0}_{\pm}}) \quad \mbox{for} \, n=2 \, . $$

We start to prove (\ref{3.6}) by estimating the first summand of $\mathcal{M}$ in (\ref{3.3}). \\
{\bf Claim 1:} $$\| \langle \nabla \rangle^{-1} \mathcal{A}_{(\pm_1,\pm_2)} (A_{0\pm_1},\phi_{\pm_2}) \|_{X^{s,-\frac{1}{2}++}_{\pm}} \lesssim \|A_{0\pm_1}\|_{X^{r,1-\epsilon_0}_{\pm_1}} \|\phi_{\pm_2}\|_{X^{s,\frac{1}{2}+}_{\pm 2}} \, . $$
This claim is equivalent to
\begin{align*}
I & := \left| \int \frac{a_{(\pm_1,\pm_2)} (\eta,\xi-\eta) \widehat{w}(\tau,\xi)}{\langle \xi \rangle^{1-s} \langle \tau \pm |\xi| \rangle^{\frac{1}{2}--}} \frac{\widehat{u}(\lambda,\eta)}{\langle \eta \rangle^r \langle \lambda \pm_1 |\eta| \rangle^{1-\epsilon_0}} \frac{\widehat{v}(\tau - \lambda,\xi - \eta)}{\langle \xi - \eta \rangle^s \langle \tau - \lambda \pm_2 |\xi - \eta| \rangle^{\frac{1}{2}+}} \right| \\
& \lesssim \|u\|_{L^2} \|v\|_{L^2} \|w\|_{L^2} \, .
\end{align*}
Here and in the rest of the paper we may assume that the Fourier transforms are nonnegative.
We use the following estimate, which follows from (\ref{2.3}) and (\ref{angle1}):
\begin{align*}
&|a_{(\pm_1,\pm_2)}(\eta,\xi - \eta)|  \lesssim m + |\xi - \eta| \angle(\pm_1 \eta,\pm_2 (\xi-\eta)) \\
& \lesssim m + |\xi-\eta| \left(\left(\frac{\langle \tau \pm |\xi|\rangle}{\min(\langle \eta \rangle, \langle \xi - \eta \rangle)} \right)^{\frac{1}{2}--} + \frac{\langle \lambda \pm_1 |\eta| \rangle^{\frac{1}{2}} + \langle \tau - \lambda \pm_2 |\xi - \eta| \rangle^{\frac{1}{2}}}{\min(\langle \eta \rangle,\langle \xi-\eta \rangle)^{\frac{1}{2}}} \right) \, .
\end{align*}
We first consider the second summand and distinguish three cases.\\
{\bf Case 1:} $(\frac{\langle \tau \pm |\xi|\rangle}{\min(\langle \eta \rangle, \langle \xi - \eta \rangle)})^{\frac{1}{2}--}$ dominant.\\
{\bf Case 1a:} $|\xi - \eta| \le |\eta|$.\\
We obtain
$$I  \lesssim  \int \frac{ \widehat{w}(\tau,\xi)}{\langle \xi \rangle^{1-s}} \frac{\widehat{u}(\lambda,\eta)}{\langle \eta \rangle^r \langle \lambda \pm_1 |\eta| \rangle^{1-\epsilon_0}} \frac{\widehat{v}(\tau - \lambda,\xi - \eta)}{\langle \xi - \eta \rangle^{s-\frac{1}{2}--} \langle \tau - \lambda \pm_2 |\xi - \eta| \rangle^{\frac{1}{2}+}} \, . $$
The desired estimate follows by application of Theorem \ref{Theorem3} with the choice
$s_0=1-s$ , $b_0=0$ , $s_1=r$ , $b_1=1-\epsilon_0$ , $s_2=s-\frac{1}{2}--$ , $b_2=\frac{1}{2}+$. The assumptions of the theorem are fulfilled under our choice of $s$ and $r$.\\
{\bf Case 1b:} $|\xi-\eta| \ge |\eta|$.\\
In this case we obtain
$$I  \lesssim  \int \frac{ \widehat{w}(\tau,\xi)}{\langle \xi \rangle^{1-s}} \frac{\widehat{u}(\lambda,\eta)}{\langle \eta \rangle^{r+\frac{1}{2}--} \langle \lambda \pm_1 |\eta| \rangle^{1-\epsilon_0}} \frac{\widehat{v}(\tau - \lambda,\xi - \eta)}{\langle \xi - \eta \rangle^{s-1} \langle \tau - \lambda \pm_2 |\xi - \eta| \rangle^{\frac{1}{2}+}} \, , $$
which allows an application of Theorem \ref{Theorem3} with the choice
$s_0=1-s$ , $b_0=0$ , $s_1=r+\frac{1}{2}--$ , $b_1=1-\epsilon_0$ , $s_2=s-1$ , $b_2=\frac{1}{2}+$.\\
{\bf Case 2:} $(\frac{\langle \lambda \pm_1 |\eta|\rangle}{\min(\langle \eta \rangle, \langle \xi - \eta \rangle)})^{\frac{1}{2}}$ dominant.\\
{\bf Case 2a:} $|\xi - \eta| \le |\eta|$.\\
We obtain 
$$I  \lesssim  \int \frac{ \widehat{w}(\tau,\xi)}{\langle \xi \rangle^{1-s} \langle \tau \pm |\xi| \rangle^{\frac{1}{2}--}} \frac{\widehat{u}(\lambda,\eta)}{\langle \eta \rangle^r \langle \lambda \pm_1 |\eta| \rangle^{\frac{1}{2}-\epsilon_0}} \frac{\widehat{v}(\tau - \lambda,\xi - \eta)}{\langle \xi - \eta \rangle^{s-\frac{1}{2}} \langle \tau - \lambda \pm_2 |\xi - \eta| \rangle^{\frac{1}{2}+}} \, . $$
The desired estimate follows by application of Theorem \ref{Theorem3} with the choice
$s_0=1-s$ , $b_0=\frac{1}{2}--$ , $s_1=r$ , $b_1=\frac{1}{2}-\epsilon_0$ , $s_2=s-\frac{1}{2}$ , $b_2=\frac{1}{2}+$.\\
{\bf Case 2b:} $|\xi - \eta| \ge |\eta|$.\\
We obtain
$$I  \lesssim  \int \frac{ \widehat{w}(\tau,\xi)}{\langle \xi \rangle^{1-s} \langle \tau \pm |\xi| \rangle^{\frac{1}{2}--}} \frac{\widehat{u}(\lambda,\eta)}{\langle \eta \rangle^{r +\frac{1}{2}} \langle \lambda \pm_1 |\eta| \rangle^{\frac{1}{2}-\epsilon_0}} \frac{\widehat{v}(\tau - \lambda,\xi - \eta)}{\langle \xi - \eta \rangle^{s-1} \langle \tau - \lambda \pm_2 |\xi - \eta| \rangle^{\frac{1}{2}+}} \, . $$
We apply Theorem \ref{Theorem3} with 
$s_0=1-s$ , $b_0=\frac{1}{2}--$ , $s_1=r+\frac{1}{2}$ , $b_1=\frac{1}{2}-\epsilon_0$ , $s_2=s-1$ , $b_2=\frac{1}{2}+$. \\
{\bf Case 3:} $(\frac{\langle \tau -\lambda \pm_2 |\xi - \eta|\rangle}{\min(\langle \eta \rangle, \langle \xi - \eta \rangle)})^{\frac{1}{2}}$ dominant.\\
{\bf Case 3a:} $|\xi - \eta| \le |\eta|$.\\
We obtain 
$$I  \lesssim  \int \frac{ \widehat{w}(\tau,\xi)}{\langle \xi \rangle^{1-s} \langle \tau \pm |\xi| \rangle^{\frac{1}{2}--}} \frac{\widehat{u}(\lambda,\eta)}{\langle \eta \rangle^r \langle \lambda \pm_1 |\eta| \rangle^{1-\epsilon_0}} \frac{\widehat{v}(\tau - \lambda,\xi - \eta)}{\langle \xi - \eta \rangle^{s-\frac{1}{2}}} $$
and use Theorem \ref{Theorem3} with 
$s_0=1-s$ , $b_0=\frac{1}{2}--$ , $s_1=r$ , $b_1=1-\epsilon_0$ , $s_2=s-1$,  $b_2=0$.\\
{\bf Case 3b:} $|\xi - \eta| \ge |\eta|$.\\
We obtain
$$I  \lesssim  \int \frac{ \widehat{w}(\tau,\xi)}{\langle \xi \rangle^{1-s} \langle \tau \pm |\xi| \rangle^{\frac{1}{2}--}} \frac{\widehat{u}(\lambda,\eta)}{\langle \eta \rangle^{r +\frac{1}{2}} \langle \lambda \pm_1 |\eta| \rangle^{1-\epsilon_0}} \frac{\widehat{v}(\tau - \lambda,\xi - \eta)}{\langle \xi - \eta \rangle^{s-1}}  $$
and apply Theorem \ref{Theorem3} with 
$s_0=1-s$ , $b_0=\frac{1}{2}--$ , $s_1=r+\frac{1}{2}$ , $b_1=1-\epsilon_0$ , $s_2=s-1$ , $b_2=0$. \\
Next we consider the first summand which is easier to handle. We obtain the estimate
$$I  \lesssim m \int \frac{ \widehat{w}(\tau,\xi)}{\langle \xi \rangle^{1-s} \langle \tau \pm |\xi| \rangle^{\frac{1}{2}--}} \frac{\widehat{u}(\lambda,\eta)}{\langle \eta \rangle^r \langle \lambda \pm_1 |\eta| \rangle^{1-\epsilon_0}} \frac{\widehat{v}(\tau - \lambda,\xi - \eta)}{\langle \xi - \eta \rangle^s \langle \tau - \lambda \pm_2 |\xi - \eta| \rangle^{\frac{1}{2}+}} \, . $$
and use Theorem \ref{Theorem3} with 
$s_0=1-s$ , $b_0=\frac{1}{2}--$ , $s_1=r$ , $b_1=1-\epsilon_0$ , $s_2=s$ , $b_2=\frac{1}{2}+$.\\
{\bf Claim 2a:} For $n=3$ we have
$$\| \langle \nabla \rangle^{-1} \mathcal{A}_{(\pm_1,\pm_2)} (A_{0\pm_1},\phi_{\pm_2}) \|_{X^{s,-\frac{1}{2}++}_{\pm}} \lesssim \|\nabla A_{0\pm_1}\|_{X^{r-1,1-\epsilon_0}_{\pm_1}} \|\phi_{\pm_2}\|_{X^{s,\frac{1}{2}+}} \, . $$
For large frequencies $|\eta| \ge 1$ of $A$ this is contained in claim 1. If $|\eta|  \le 1$ we have to replace $I$ by $I'$ where $\langle \eta \rangle^r$ is replaced by $\langle \eta \rangle^{r-1} |\eta|$, which is equivalent to $ \langle \eta \rangle |\eta|$, thus we obtain 
$$I'  \lesssim  \int \frac{ \widehat{w}(\tau,\xi)}{\langle \xi \rangle^{1-s} \langle \tau \pm |\xi| \rangle^{\frac{1}{2}--}} \frac{\widehat{u}(\lambda,\eta)}{\langle \eta \rangle |\eta| \langle \lambda \pm_1 |\eta| \rangle^{1-\epsilon_0}} \frac{\widehat{v}(\tau - \lambda,\xi - \eta)}{\langle \xi - \eta \rangle^{s-1} \langle \tau - \lambda \pm_2 |\xi - \eta| \rangle^{\frac{1}{2}+}} \, . $$
We have $\langle \xi \rangle \sim \langle \xi - \eta \rangle$, thus
\begin{align*}
I' & \lesssim  \int \frac{ \widehat{w}(\tau,\xi)}{ \langle \tau \pm |\xi| \rangle^{\frac{1}{2}--}}
 \frac{\widehat{u}(\lambda,\eta)}{\langle \eta \rangle |\eta| \langle \lambda \pm_1 |\eta| \rangle^{1-\epsilon_0}}
  \frac{\widehat{v}(\tau - \lambda,\xi - \eta)}{ \langle \tau - \lambda \pm_2 |\xi - \eta| \rangle^{\frac{1}{2}+}} \\
 & \lesssim \|w\|_{L^2_{xt}}
\| {\mathcal F}^{-1} (\frac{\widehat{u}(\lambda,\eta)}{\langle \eta \rangle |\eta| \langle \lambda \pm_1 |\eta|\rangle^{1-\epsilon_0}})\|_{L^{\infty}_t L_x^{\infty}} 
\|v\|_{L^2_{xt}} \, .  
 \end{align*}
Now using $H^{\frac{1}{2}+,6}_x \subset L^{\infty}_x$ and $\dot{H}^{1,2}_x \subset L_x^6$ we obtain
$$\| {\mathcal F}^{-1} (\frac{\widehat{u}(\lambda,\eta)}{\langle \eta \rangle |\eta| \langle \lambda \pm_1 |\eta|\rangle^{1-\epsilon_0}})\|_{L^{\infty}_t L_x^{\infty}} \lesssim
\| {\mathcal F}^{-1} (\frac{\widehat{u}(\lambda,\eta)}{\langle \eta \rangle^{\frac{1}{2}-}  \langle \lambda \pm_1 |\eta|\rangle^{1-\epsilon_0}})\|_{L^{\infty}_t L_x^2} \lesssim \|u\|_{L^2_{xt}} \, ,
$$
thus the desired estimate.\\
{\bf Claim 2b:} For $n=2$ we have
$$\| \langle \nabla \rangle^{-1} \mathcal{A}_{(\pm_1,\pm_2)} (A_{0\pm_1},\phi_{\pm_2}) \|_{X^{s,-\frac{1}{2}++}_{\pm}} \lesssim \||\nabla|^{1-\epsilon} A_{0\pm_1}\|_{X^{r-1+\epsilon,1-\epsilon_0}_{\pm_1}} \|\phi_{\pm_2}\|_{X^{s,\frac{1}{2}+}_{\pm_2}} \, . $$
We estimate similarly as in the 3D case using $H^{\epsilon+,\frac{2}{\epsilon}}_x \subset L^{\infty}_x$ and $\dot{H}^{1-\epsilon,2}_x \subset L^{\frac{2}{\epsilon}}_x$ to obtain
\begin{align*}
\| {\mathcal F}^{-1} (\frac{\widehat{u}(\lambda,\eta)}{\langle \eta \rangle |\eta|^{1-\epsilon} \langle \lambda \pm_1 |\eta|\rangle^{1-\epsilon_0}})\|_{L^{\infty}_t L_x^{\infty}} &\lesssim
\| {\mathcal F}^{-1} (\frac{\widehat{u}(\lambda,\eta)}{\langle \eta \rangle^{1-\epsilon-}  \langle \lambda \pm_1 |\eta|\rangle^{1-\epsilon_0}})\|_{L^{\infty}_t L_x^2} \\
& \lesssim \|u\|_{L^2_{xt}} \, ,
\end{align*}
giving the desired estimate.

The second summand $2i P_2$ of $\mathcal{M}$ in (\ref{3.3}) can be handled in the same way as the first summand before, because its symbol satisfies (\ref{2.0}) and therefore also the estimate (\ref{2.3}) without the term $m$.

In order to complete the proof of (\ref{3.6}) it remains to consider the cubic last summand of $\mathcal{M}$ in (\ref{3.3}).\\
{\bf Claim 3:} $$\| \langle \nabla \rangle^{-1}(A^{\mu}_{\pm_1} A^{\mu}_{\pm_2} \phi_{\pm_3})\|_{X^{s,-\frac{1}{2}++}_{\pm}} \lesssim \|A^{\mu}_{\pm_1}\|_{X^{r,1-\epsilon_0}_{\pm_1}} \|A^{\mu}_{\pm_2}\|_{X^{r,1-\epsilon_0}_{\pm_2}} \|\phi_{\pm_3}\|_{X^{s,\frac{1}{2}+}_{\pm_3}} \, . $$
In a first step we obtain for $l=0+$ and $l=-\frac{1}{4}+\delta+$ in the case $n=3$ and $n=2$, respectively:
$$\| \langle \nabla \rangle^{-1}(A^{\mu}_{\pm_1} A^{\mu}_{\pm_2} \phi_{\pm_3})\|_{X^{s,-\frac{1}{2}++}_{\pm}} \lesssim \|A^{\mu}_{\pm_1} A^{\mu}_{\pm_2} \|_{X^{l,0}_{\pm}} \|\phi_{\pm_3}\|_{X^{s,\frac{1}{2}+}_{\pm_3}} $$
by use of Theorem \ref{Theorem3} with $s_0 = 1-s$ , $b_0=\frac{1}{2}--$ , $s_1=l$ , $b_1=0$ , $s_2=s$ , $b_2=\frac{1}{2}+$. \\
In the case $n=3$ we use Strichartz' estimate (cf. \cite{GV}) and obtain
$$\|A^{\mu}_{\pm_1} A^{\mu}_{\pm_2} \|_{X^{0+,0}_{\pm}} \lesssim \|A^{\mu}_{\pm_1}\|_{L^4_t H^{0+,4}_x} \|A^{\mu}_{\pm_2}\|_{L^4_t H^{0+,4}_x}
\lesssim \|A^{\mu}_{\pm_1}\|_{X^{\frac{1}{2}+,\frac{1}{2}+}_{\pm_1}} \|A^{\mu}_{\pm_2}\|_{X^{\frac{1}{2}+,\frac{1}{2}+}_{\pm_2}} \, , $$
which gives the claimed bound.\\
In the case $n=2$ we use Theorem \ref{Theorem3} with $s_0 = \frac{1}{4}-\delta-$ , $b_0 =0$ , $s_1=s_2=\frac{1}{4}+\delta$, $b_1=b_2=1-\epsilon_0$ which gives the estimate
$$\|A^{\mu}_{\pm_1} A^{\mu}_{\pm_2} \|_{X^{-\frac{1}{4}+\delta+,0}_{\pm}}
\lesssim \|A^{\mu}_{\pm_1}\|_{X^{\frac{1}{4}+\delta,1-\epsilon_0}_{\pm_1}} \|A^{\mu}_{\pm_2}\|_{X^{\frac{1}{4}+\delta,1-\epsilon_0}_{\pm_2}} \, , $$
and thus the claimed bound.\\
{\bf Claim 4a:} In the case $n=3$ the following estimate holds:
$$\| \langle \nabla \rangle^{-1}(A^{\mu}_{\pm_1} A^{\mu}_{\pm_2} \phi_{\pm_3})\|_{X^{s,-\frac{1}{2}++}_{\pm}} \lesssim \| \nabla A^{\mu}_{\pm_1}\|_{X^{r-1,1-\epsilon_0}_{\pm_1}} \| \nabla A^{\mu}_{\pm_2}\|_{X^{r-1,1-\epsilon_0}_{\pm_2}} \|\phi_{\pm_3}\|_{X^{s,\frac{1}{2}+}_{\pm_3}} \, . $$
It remains to consider the case of low frequencies in at least one of the factors $A^{\mu}_{\pm_1}$ or $A^{\mu}_{\pm_2}$. By Sobolev's embedding we obtain
\begin{align*}
\| \langle \nabla \rangle^{-1}(A^{\mu}_{\pm_1} A^{\mu}_{\pm_2} \phi_{\pm_3})\|_{X^{s,-\frac{1}{2}++}_{\pm}} &\lesssim \|A^{\mu}_{\pm_1} A^{\mu}_{\pm_2} \phi_{\pm_3}\|_{L^2_t H^{s-1}_x} 
 \lesssim \|A^{\mu}_{\pm_1} A^{\mu}_{\pm_2} \phi_{\pm_3}\|_{L^2_t L^{\frac{12}{7}+}_x} \\
 &\lesssim \|A^{\mu}_{\pm_1}\|_{L^4_t L^6_x} \| A^{\mu}_{\pm_2}\|_{L^4_t L^6_x} \| \phi_{\pm_3}\|_{L^{\infty}_t L^{4+}_x} \, ,
\end{align*}
which in the case where both frequencies are small is simply estimated by
\begin{align*}
&\|\nabla A^{\mu}_{\pm_1}\|_{L^4_t L^2_x} \|\nabla A^{\mu}_{\pm_2}\|_{L^4_t L^2_x} \| \phi_{\pm_3}\|_{L^{\infty}_t H^{\frac{3}{4}+}_x} \\
&\lesssim \|\nabla A^{\mu}_{\pm_1}\|_{X^{r-1,\frac{1}{2}+}_{\pm_1}} \|\nabla A^{\mu}_{\pm_2}\|_{X^{r-1,\frac{1}{2}+}_{\pm_2}} \| \phi_{\pm_3}\|_{X^{s,\frac{1}{2}+}_{\pm_3}} 
\end{align*}
as desired, whereas if the frequency of, say, $A^{\mu}_{\pm_1}$ is $\le 1$ and the frequency of $A^{\mu}_{\pm_2}$ is $\ge 1$, we use Strichartz' estimate (cf. \cite{GV}) and Sobolev's embedding to obtain the bound
\begin{align*}
 &\|A^{\mu}_{\pm_1}\|_{L^4_t L^{12}_x} \| A^{\mu}_{\pm_2}\|_{L^4_t L^4_x} \| \phi_{\pm_3}\|_{L^{\infty}_t L^{4+}_x} \lesssim \|A^{\mu}_{\pm_1}\|_{L^4_t H^{\frac{1}{4},6}_x} \| A^{\mu}_{\pm_2}\|_{X^{\frac{1}{2}+,\frac{1}{2}+}_x} \| \phi_{\pm_3}\|_{X^{\frac{3}{4}+,\frac{1}{2}+}_x}\\
&\lesssim  \|\nabla A^{\mu}_{\pm_1}\|_{L^4_t H^{\frac{1}{4},2}_x} \| \nabla A^{\mu}_{\pm_2}\|_{X^{-\frac{1}{2}+,\frac{1}{2}+}_{\pm_2}} \| \phi_{\pm_3}\|_{X^{\frac{3}{4}+,\frac{1}{2}+}_{\pm_3}} \\
&\lesssim \|\nabla A^{\mu}_{\pm_1}\|_{X^{r-1,1-\epsilon_0}_{\pm_1}} \| \nabla A^{\mu}_{\pm_2}\|_{X^{-\frac{1}{2}+,\frac{1}{2}+}_{\pm_2}} \| \phi_{\pm_3}\|_{X^{\frac{3}{4}+,\frac{1}{2}+}_{\pm_3}} \, .
\end{align*}
{\bf Claim 4b:} In the case $n=2$ the following estimate holds:
\begin{align*}
&\| \langle \nabla \rangle^{-1}(A^{\mu}_{\pm_1} A^{\mu}_{\pm_2} \phi_{\pm_3})\|_{X^{s,-\frac{1}{2}++}_{\pm}} \\
&\lesssim \| |\nabla|^{1-\epsilon} A^{\mu}_{\pm_1}\|_{X^{r-1+\epsilon,1-\epsilon_0}_{\pm_1}} \| |\nabla|^{1-\epsilon} A^{\mu}_{\pm_2}\|_{X^{r-1+\epsilon,1-\epsilon_0}_{\pm_2}} \|\phi_{\pm_3}\|_{X^{s,\frac{1}{2}+}_{\pm_3}} \, . 
\end{align*}
It remains to consider the case of low frequencies in at least one of the factors $A^{\mu}_{\pm_1}$ or $A^{\mu}_{\pm_2}$. We crudely estimate
$$\| \langle \nabla \rangle^{-1}(A^{\mu}_{\pm_1} A^{\mu}_{\pm_2} \phi_{\pm_3})\|_{X^{s,-\frac{1}{2}++}_{\pm}}  \lesssim \|A^{\mu}_{\pm_1} A^{\mu}_{\pm_2} \phi_{\pm_3}\|_{L^2_{xt}} \, . $$
If both frequencies are $\le 1$ we obtain the bound
\begin{align*}
&\|A^{\mu}_{\pm_1}\|_{L^4_t L^{\frac{2}{\epsilon}}_x} \|A^{\mu}_{\pm_2}\|_{L^4_t L^{\frac{2}{\epsilon}}_x} \|\phi_{\pm_3}\|_{L^{\infty}_t L^{\frac{2}{1-2\epsilon}}_x} \\
&\lesssim \||\nabla|^{1-\epsilon} A^{\mu}_{\pm_1}\|_{L^4_t L^2_x} 
\||\nabla|^{1-\epsilon} A^{\mu}_{\pm_2}\|_{L^4_t L^2_x}
 \|\phi_{\pm_3}\|_{X^{2\epsilon,\frac{1}{2}+}_{\pm_3}} \\
 & \lesssim \||\nabla|^{1-\epsilon} A^{\mu}_{\pm_1}\|_{X^{r-1+\epsilon,\frac{1}{2}+}_{\pm_1}} 
\||\nabla|^{1-\epsilon} A^{\mu}_{\pm_2}\|_{X^{r-1+\epsilon,\frac{1}{2}+}_{\pm_2}}
 \|\phi_{\pm_3}\|_{X^{s,\frac{1}{2}+}_{\pm_3}} \, .
\end{align*}
If the frequencies of, say, $A^{\mu}_{\pm_1}$ are $\le 1$ and the frequencies of $A^{\mu}_{\pm_2}$ are $\ge 1$, we obtain the bound
\begin{align*}
&\|A^{\mu}_{\pm_1}\|_{L^{\infty}_t L^{\frac{2}{\epsilon}}_x} \|A^{\mu}_{\pm_2}\|_{L^2_t L^{\frac{8}{3-4\epsilon}}_x} \|\phi_{\pm_3}\|_{L^{\infty}_t L^8_x} \\
&\lesssim \||\nabla|^{1-\epsilon} A^{\mu}_{\pm_1}\|_{L^{\infty}_t L^2_x} 
\| A^{\mu}_{\pm_2}\|_{L^2_t H^{\frac{1}{4}+\epsilon}_x}
 \|\phi_{\pm_3}\|_{L^{\infty}_t H^{\frac{3}{4}}_x} \\
 & \lesssim \||\nabla|^{1-\epsilon} A^{\mu}_{\pm_1}\|_{X^{r-1+\epsilon,\frac{1}{2}+}_{\pm_1}} 
\||\nabla|^{1-\epsilon} A^{\mu}_{\pm_2}\|_{X^{r-1+\epsilon,\frac{1}{2}+}_{\pm_2}}
 \|\phi_{\pm_3}\|_{X^{2\epsilon,\frac{1}{2}+}_{\pm_3}} \, .
\end{align*}
This completes the proof of (\ref{3.6}).

Next we prove (\ref{3.7}) and (\ref{3.8}).\\
{\bf Claim 5a:} In the case $n=3$ the following estimate holds:
$$ \| |\nabla|^{-1} (\phi_{\pm_1} \langle \nabla \rangle \phi_{\pm_2})\|_{X^{r,-\epsilon_0 +}_{\pm}} \lesssim \|\phi_{\pm_1}\|_{X^{s,\frac{1}{2}+}_{\pm_1}}
\|\phi_{\pm_2}\|_{X^{s,\frac{1}{2}+}_{\pm_2}} \,. $$
For frequencies of the product which are $\ge 1$ the left hand side is equivalent to
$\|\phi_{\pm_1} \langle \nabla \rangle \phi_{\pm_2}\|_{X^{r-1,-\epsilon_0 +}_{\pm}}$ and the claimed estimate follows by use of Theorem \ref{Theorem3} with $s_0 = 1-r$ , $b_0 = \epsilon_0 +$ , $s_1=s_2=s$ , $b_1 = b_2 = \frac{1}{2}+$ .\\
For low frequencies $\le 1$ of the product we have frequencies of the factors which are equivalent. Thus we obtain
\begin{align*}
\| |\nabla|^{-1}(\phi_{\pm_1} \langle \nabla \rangle 
\phi_{\pm_2})\|_{X^{r,-\epsilon_0+}_{\pm}}
&\lesssim \| |\nabla|^{-1}(\langle \nabla \rangle^{\frac{1}{2}} \phi_{\pm_1} \langle \nabla \rangle^{\frac{1}{2}} \phi_{\pm_2}) \|_{L^2_{xt}} \\
\lesssim \|\langle \nabla \rangle^{\frac{1}{2}} \phi_{\pm_1} \langle \nabla \rangle^{\frac{1}{2}} \phi_{\pm_2} \|_{L^2_t L^{\frac{6}{5}}_x} 
&\lesssim \|\langle \nabla \rangle^{\frac{1}{2}} \phi_{\pm_1}\|_{L^4_t L^{\frac{12}{5}}_x} \| \langle \nabla \rangle^{\frac{1}{2}} \phi_{\pm_2} \|_{L^4_t L^{\frac{12}{5}}_x} \\
\lesssim \|\phi_{\pm_1}\|_{L^4_t H^{\frac{3}{4}}_x} \|\phi_{\pm_2}\|_{L^4_t H^{\frac{3}{4}}_x} &\lesssim \|\phi_{\pm_1}\|_{X^{s,\frac{1}{2}+}_{\pm_1}} \|\phi_{\pm_2}\|_{X^{s,\frac{1}{2}+}_{\pm_2}} \, . 
\end{align*}
{\bf Claim 5b:} In the case $n=2$ the following estimate holds:
$$ \| |\nabla|^{-1+\epsilon_1} (\phi_{\pm_1} \langle \nabla \rangle \phi_{\pm_2})\|_{X^{r-\epsilon_1,-\epsilon_0 +}_{\pm}} \lesssim \|\phi_{\pm_1}\|_{X^{s,\frac{1}{2}+}_{\pm_1}}
\|\phi_{\pm_2}\|_{X^{s,\frac{1}{2}+}_{\pm_2}} \,. $$
For frequencies of the product $\ge 1$ we estimate similarly as in the case $n=3$.\\
For low frequencies $\le 1$ of the product we obtain
\begin{align*}
\| |\nabla|^{-1+\epsilon_1}(\phi_{\pm_1} \langle \nabla \rangle 
\phi_{\pm_2})\|_{X^{r-\epsilon_1,-\epsilon_0+}_{\pm}}
&\lesssim \| |\nabla|^{-1+\epsilon_1}(\langle \nabla \rangle^{\frac{1}{2}} \phi_{\pm_1} \langle \nabla \rangle^{\frac{1}{2}} \phi_{\pm_2}) \|_{L^2_{xt}} \\
\lesssim \|\langle \nabla \rangle^{\frac{1}{2}} \phi_{\pm_1} \langle \nabla \rangle^{\frac{1}{2}} \phi_{\pm_2} \|_{L^2_t L^{\frac{2}{2-\epsilon_1}}_x} 
&\lesssim \|\langle \nabla \rangle^{\frac{1}{2}} \phi_{\pm_1}\|_{L^4_t L^{\frac{4}{2-\epsilon_1}}_x} \| \langle \nabla \rangle^{\frac{1}{2}} \phi_{\pm_2} \|_{L^4_t L^{\frac{4}{2-\epsilon_1}}_x} \\
\lesssim \|\phi_{\pm_1}\|_{L^4_t H^{\frac{1+\epsilon_1}{2}}_x} \|\phi_{\pm_2}\|_{L^4_t H^{\frac{1+\epsilon_1}{2}}_x} &\lesssim \|\phi_{\pm_1}\|_{X^{s,\frac{1}{2}+}_{\pm_1}} \|\phi_{\pm_2}\|_{X^{s,\frac{1}{2}+}_{\pm_2}} \, . 
\end{align*}
{\bf Claim 6:} In dimension $n=3$ the following estimate holds:
$$ \| |\nabla |^{-1}(A_{\pm_3} \phi_{\pm_1} \phi_{\pm_2}) \|_{X^{r,-\epsilon_0 +}_{\pm}} \lesssim \| \nabla A_{\pm_3}\|_{X^{r-1,1-\epsilon_0}_{\pm_3}} \|\phi_{\pm_1} \|_{X^{s,\frac{1}{2}+}_{\pm_1}} \|\phi_{\pm_2} \|_{X^{s,\frac{1}{2}+}_{\pm_2}} \, . $$
We first assume frequencies $\ge 1$ of the product and $A_{\pm_3}$ and use Theorem \ref{Theorem3} with $s_0 = 1-r$ , $b_0 = \epsilon_0 -$ , $s_1 = r$ , $ b_1 = 1-\epsilon_0$ , $s_2 = \frac{1}{2} - \epsilon_0 +$ , $ b_2 = 0$ to obtain
$$ \|A_{\pm_3} \phi_{\pm_1} \phi_{\pm_2} \|_{X^{r-1,-\epsilon_0 +}_{\pm}} \lesssim \| A_{\pm_3}\|_{X^{r,1-\epsilon_0}_{\pm_3}} \|\phi_{\pm_1} \phi_{\pm_2} \|_{X^{\frac{1}{2}-\epsilon_0+,0}_{\pm}} \, . $$
Next we have
$$\|\phi_{\pm_1} \phi_{\pm_2} \|_{X^{\frac{1}{2}-\epsilon_0+,0}_{\pm}}
\lesssim \|\phi_{\pm_1} \|_{X^{\frac{3}{4}+\delta,\frac{1}{2}+}_{\pm_1}}
 \|\phi_{\pm_2} \|_{X^{\frac{3}{4}+\delta,\frac{1}{2}+}_{\pm_2}}  $$
by Theorem \ref{Theorem3} with $s_0= -\frac{1}{2}+\epsilon_0-$ , $b_0=0$ , $s_1=s_2=\frac{3}{4}+\delta$ , $b_1=b_2=\frac{1}{2}+$, which gives the claimed estimate.\\
Next we assume frequencies $\le 1$ of the product and frequencies $\ge 1$ of $A_{\pm_3}$ which is handled as follows:
\begin{align*}
&\| |\nabla |^{-1}(A_{\pm_3} \phi_{\pm_1} \phi_{\pm_2}) \|_{X^{r,-\epsilon_0 +}_{\pm}}
\lesssim \|A_{\pm_3} \phi_{\pm_1} \phi_{\pm_2} \|_{L^2_t L^{\frac{6}{5}}_x} \\
&\lesssim \|A_{\pm_3}\|_{L^{\infty}_t L^3_x} \|\phi_{\pm_1} \|_{L^4_t L^4_x}  \|\phi_{\pm_2} \|_{L^4_t L^4_x}
 \lesssim  \|A_{\pm_3}\|_{L^{\infty}_t H^{\frac{1}{2}}_x} \|\phi_{\pm_1} \|_{L^4_t H^{\frac{3}{4}}_x}  \|\phi_{\pm_2} \|_{L^4_t H^{\frac{3}{4}}_x} \\
&\lesssim \|A_{\pm_3}\|_{X^{\frac{1}{2},\frac{1}{2}+}_{\pm_3}} \|\phi_{\pm_1} \|_{X^{\frac{3}{4},\frac{1}{2}+}_{\pm_1}}  \|\phi_{\pm_2} \|_{X^{\frac{3}{4},\frac{1}{2}+}_{\pm_2}} \, .
\end{align*}
Finally we have to treat frequencies $\le 1$ of $A_{\pm_3}$ by Sobolev's embedding:
\begin{align*}
&\| |\nabla|^{-1}(A_{\pm_3} \phi_{\pm_1} \phi_{\pm_2} ) \|_{X^{r,-\epsilon_0 +}_{\pm}} \lesssim \|A_{\pm_3} \phi_{\pm_1} \phi_{\pm_2}\|_{L^2_t H^{r,\frac{6}{5}}_x} \\
&\lesssim \|\langle \nabla \rangle^r A_{\pm_3} \phi_{\pm_1} \phi_{\pm_2}\|_{L^2_t L^{\frac{6}{5}}_x} + \| A_{\pm_3} \langle \nabla \rangle^r \phi_{\pm_1} \phi_{\pm_2}\|_{L^2_t L^{\frac{6}{5}}_x} +
\| A_{\pm_3} \phi_{\pm_1} \langle \nabla \rangle^r \phi_{\pm_2}\|_{L^2_t L^{\frac{6}{5}}_x} \\
&\lesssim \|\langle \nabla \rangle^r A_{\pm_3}\|_{L^2_t L^6_x} \| \phi_{\pm_1} \|_{L^{\infty}_t L^3_x} \| \phi_{\pm_2}\|_{L^{\infty}_t L^3_x} \\
& + \| A_{\pm_3} \|_{L^2_t L^6_x} (\| \langle \nabla \rangle^r \phi_{\pm_1}\|_{L^2_t L^{\frac{12}{5}}_x} \| \phi_{\pm_2}\|_{L^2_t L^4_x} + \| \phi_{\pm_1}\|_{L^2_t L^4_x}
\|\langle \nabla \rangle^r \phi_{\pm_2}\|_{L^2_t L^{\frac{12}{5}}_x}) \\
&\lesssim \| \nabla  A_{\pm_3}\|_{L^2_t L^2_x} \| \phi_{\pm_1} \|_{L^{\infty}_t H^{\frac{1}{2}}_x} \| \phi_{\pm_2}\|_{L^{\infty}_t H^{\frac{1}{2}}_x} \\
& + \|\nabla  A_{\pm_3} \|_{L^2_t L^2_x} (\| \langle \nabla \rangle^r \phi_{\pm_1}\|_{L^2_t H^{\frac{1}{4}}_x} \| \phi_{\pm_2}\|_{L^2_t H^{\frac{3}{4}}_x} + \| \phi_{\pm_1}\|_{L^2_t H^{\frac{3}{4}}_x}
\|\langle \nabla \rangle^r \phi_{\pm_2}\|_{L^2_t H^{\frac{1}{4}}_x}) \\
&\lesssim  \| \nabla A_{\pm_3}\|_{X^{r-1,1-\epsilon_0}_{\pm_3}} \|\phi_{\pm_1} \|_{X^{s,\frac{1}{2}+}_{\pm_1}} \|\phi_{\pm_2} \|_{X^{s,\frac{1}{2}+}_{\pm_2}} \, .
\end{align*}
{\bf Claim 7:} In dimension $n=2$ the following estimate holds:
\begin{align*} &\| |\nabla |^{-1+\epsilon_1}(A_{\pm_3} \phi_{\pm_1} \phi_{\pm_2}) \|_{X^{r-\epsilon_1,-\epsilon_0 +}_{\pm}} \\
&\lesssim \| |\nabla|^{1-\epsilon} A_{\pm_3}\|_{X^{r-1+\epsilon,1-\epsilon_0}_{\pm_3}} \|\phi_{\pm_1} \|_{X^{s,\frac{1}{2}+}_{\pm_1}} \|\phi_{\pm_2} \|_{X^{s,\frac{1}{2}+}_{\pm_2}} \, . \end{align*}
For frequencies $\ge 1$ of the product and $A_{\pm_3}$ we first obtain
$$\|A_{\pm_3} \phi_{\pm_1} \phi_{\pm_2} \|_{X^{r-1,-\epsilon_0 +}_{\pm}} \lesssim \|  A_{\pm_3}\|_{X^{r,1-\epsilon_0}_{\pm_3}} \|\phi_{\pm_1} \phi_{\pm_2} \|_{X_{\pm}^{0,0}}   $$
by Theorem \ref{Theorem3} with $s_0 = 1-r$ , $b_0 = \epsilon_0 -$ , $s_1=r$ , $b_1 = 1-\epsilon_0$ , $s_2=0$ , $b_2=0$. Strichartz' estimate (\cite{GV}) shows
$$\|\phi_{\pm_1} \phi_{\pm_2} \|_{X_{\pm}^{0,0}}  \lesssim \|\phi_{\pm_1}\|_{L^4_t L^4_x}
\|\phi_{\pm_2}\|_{L^4_t L^4_x} \le \|\phi_{\pm_1}\|_{X^{\frac{1}{2}+,\frac{1}{2}+}_{\pm_1}}
\|\phi_{\pm_2}\|_{X^{\frac{1}{2}+,\frac{1}{2}+}_{\pm_2}} \, , $$
thus the claimed result.\\
For frequencies $\le 1$ for the product and $\ge 1$ for $A_{\pm_3}$ we estimate as follows:
\begin{align*}
\| |\nabla |^{-1+\epsilon_1}(A_{\pm_3} \phi_{\pm_1} \phi_{\pm_2}) \|_{X^{r-\epsilon_1,-\epsilon_0 +}_{\pm}} &\lesssim \|A_{\pm_3} \phi_{\pm_1} \phi_{\pm_2}\|_{L^2_t L^{\frac{2}{2-\epsilon_1}}_x} \\
&\lesssim \|A_{\pm_3}\|_{L^{\infty}_t L^{\frac{2}{1-\epsilon_1}}_x} \|\phi_{\pm_1}\|_{L^4_t L^4_x} \|\phi_{\pm_2}\|_{L^4_t L^4_x} \\
&\lesssim \|A_{\pm_3}\|_{X^{\epsilon_1,\frac{1}{2}+}_{\pm_3}} \|\phi_{\pm_1}\|_{X^{\frac{1}{2}+,\frac{1}{2}+}_{\pm_1}} 
\|\phi_{\pm_2}\|_{X^{\frac{1}{2}+,\frac{1}{2}+}_{\pm_2}} \, ,
\end{align*}
which implies the claim, so that it remains to consider frequencies $\le 1$ of $A_{\pm_3}$. We estimate as follows:
\begin{align*} &\| |\nabla |^{-1+\epsilon_1}(A_{\pm_3} \phi_{\pm_1} \phi_{\pm_2}) \|_{X^{r-\epsilon_1,-\epsilon_0 +}_{\pm}} \\
&\lesssim \| A_{\pm_3} \phi_{\pm_1} \phi_{\pm_2} \|_{L^2_t H^{r-\epsilon_1,\frac{2}{2-\epsilon_1}}_x} \\
&\lesssim \| \langle \nabla \rangle^{r-\epsilon_1} A_{\pm_3} \phi_{\pm_1} \phi_{\pm_2} \|_{L^2_t L^{\frac{2}{2-\epsilon_1}}_x} + \|  A_{\pm_3} \langle \nabla \rangle^{r-\epsilon_1}\phi_{\pm_1} \phi_{\pm_2} \|_{L^2_t L^{\frac{2}{2-\epsilon_1}}_x} \\
&  \hspace{10em}  +\| A_{\pm_3} \phi_{\pm_1} \langle \nabla \rangle^{r-\epsilon_1}\phi_{\pm_2} \|_{L^2_t L^{\frac{2}{2-\epsilon_1}}_x} \\
& \lesssim \| \langle \nabla \rangle^{r-\epsilon_1} A_{\pm_3}\|_{L^2_t L^{\frac{2}{\epsilon}}_x} \| \phi_{\pm_1} \|_{L^{\infty}_t L^{\frac{4}{2-\epsilon-\epsilon_1}}_x} \| \phi_{\pm_2} \|_{L^{\infty}_t L^{\frac{4}{2-\epsilon -\epsilon_1}}_x} 
+ \|  A_{\pm_3} \|_{L^2_t L^{\frac{2}{\epsilon}}_x} \\
&\hspace{1em} \times (\|\langle \nabla \rangle^{r-\epsilon_1} \phi_{\pm_1} \|_{L^{\infty}_t L^2_x} \| \phi_{\pm_2} \|_{L^{\infty}_t L^{\frac{2}{1-\epsilon -\epsilon_1}}_x} +  \| \phi_{\pm_1} \|_{L^{\infty}_t L^{\frac{2}{1-\epsilon -\epsilon_1}}_x} \|\langle \nabla \rangle^{r-\epsilon_1} \phi_{\pm_2} \|_{L^{\infty}_t L^2_x})\\
& \lesssim \| |\nabla|^{1-\epsilon} A_{\pm_3}\|_{L^2_t L^2_x} \|\phi_{\pm_1}\|_{L^{\infty}_t H^{r-\epsilon_1}_x} \|\phi_{\pm_2}\|_{L^{\infty}_t H^{r-\epsilon_1}_x} \, .
 \end{align*} 
This implies the claimed result and completes the proof of (\ref{3.7}) and (\ref{3.8}) and of Theorem \ref{Theorem2}.
\end{proof}

\section{Proof of Theorem \ref{Theorem1}}
\begin{proof}
We start with the solution $(\phi_{\pm},A_{\pm})$ of (\ref{3.1}),(\ref{3.2}) given by Theorem \ref{Theorem2}. Defining $\phi:=\phi_+ + \phi_-$ , $A:=A_+ + A_-$ we immediately see that $\partial_t \phi = i \langle \nabla \rangle_m (\phi_+ - \phi_-)$ , $\partial_t A = i |\nabla|(A_+ - A_-)$, so that $\mathcal{N}$ in (\ref{3.2}) is the same as $N$ in (\ref{2.2}) and (\ref{16}). 

Moreover  by (\ref{3.2}):
\begin{align*}
\square A & = (i\partial_t - |\nabla|)(i\partial_t + |\nabla|)A_+ + (i\partial_t + |\nabla|)(i\partial_t - |\nabla|)A_- \\
& = -(i \partial_t - |\nabla|)(2|\nabla|)^{-1} \mathcal{N}(A,\phi) + (i\partial_t + |\nabla|) (2|\nabla|)^{-1} \mathcal{N}(A,\phi) = N(A,\phi) \, .
\end{align*}
Thus $A$ satisfies (\ref{16}). We also have $\phi(0) = \phi_0$ , $\partial_t \phi = \phi_1$ , $A(0) = a_0$ , $\partial_t A(0) = \dot{a}_0$.

Next we prove that the Lorenz condition $\partial^{\mu} A_{\mu} =0$ is satisfied. We define
$$
u  := \partial^{\mu} A_{\mu} = -\partial_t A_0 + \partial^j A_j \quad , \quad
u_{\pm}  := -\partial_t A_{0_{\pm}} + \partial^j A_{j_{\pm}} \, . $$
By (\ref{2.2}) we obtain
\begin{align*}
&(i\partial_t \pm |\nabla|)u_{\pm} = -\partial_t(i\partial_t \pm |\nabla|)A_{0_{\pm}} + \partial^j(i\partial_t \pm |\nabla|)A_{j_{\pm}} \\
&= (\pm 2|\nabla|)^{-1}(\partial_t(Im(-\phi \overline{i \langle \nabla \rangle_m(\phi_+ - \phi_-)}) -A_0|\phi|^2)-\partial^j(Im(-\phi\overline{\partial_j \phi}) - A_j |\phi|^2)) \\
&= (\pm 2 |\nabla|)^{-1} (Im(\phi \langle \nabla \rangle_m(-\overline{i\partial_t \phi_+} + \overline{i\partial_t \phi_-})) + Im(\phi \overline{\Delta \phi}) + \partial^{\mu}(A_{\mu} |\phi|^2)) \, .
\end{align*}
Now we have
$$ Im(\phi \overline{\Delta \phi})=Im(\phi(m^2\overline{\phi} - \langle \nabla \rangle_m \langle \nabla \rangle_m \overline{\phi})) = -Im(\phi \langle \nabla\rangle_m(\langle \nabla \rangle_m \overline{\phi}_+ + \langle \nabla \rangle_m \overline{\phi}_-)) \, , $$
so that by (\ref{3.1}) we obtain
\begin{align*}
&(i\partial_t \pm |\nabla|)u_{\pm} \\
&= (\pm 2 |\nabla|)^{-1} (Im(\phi \langle \nabla \rangle_m(-(\overline{i\partial_t + \langle \nabla \rangle_m) \phi_+} + (\overline{i\partial_t - \langle \nabla \rangle_m)\phi_-})) +  \partial^{\mu}(A_{\mu} |\phi|^2)) \\
& = (\pm 2|\nabla|)^{-1}(Im(\phi \overline{\mathcal{M}(\phi_+,\phi_-,A_+,A_-)}) + A_{\mu} \partial^{\mu}(|\phi|^2) + |\phi|^2 u) \\
& =: (\pm 2 |\nabla |)^{-1} R(A,\phi) \, .
\end{align*}
By (\ref{3.3}) and the second equation in (\ref{2.2a})
\begin{align*}
\mathcal{M}(\phi_+,\phi_-,A_+,A_-) & = 2 \sum_{\pm_1,\pm_2} \pm_2 \mathcal{A}_{(\pm_1,\pm_2)} (A_{0_{\pm_1}},\phi_{\pm_2}) + 2iP_2 + A_{\mu} A^{\mu} \phi \\
&=2((A_{0_+} + A_{0_-}) \langle \nabla \rangle_m(\phi_+ - \phi_-) \\
& \hspace{2em} + |\nabla|^{-1} \nabla(A_{0_+} - A_{0_-}) \cdot \nabla (\phi_+ + \phi_-)) +2iP_2 + A_{\mu} A^{\mu} \phi \\
& = 2i(-A_0 \partial_t \phi -  (-\Delta)^{-1} \nabla \partial_t A_0 \cdot \nabla \phi) + 2iP_2 + A_{\mu}A^{\mu} \phi \, .
\end{align*}
Now by the definition of $P_2$
\begin{align*}
 P_2=R^k(R_j A_k - R_k A_j) \partial^j \phi &= (-\Delta)^{-1} \partial^k(\partial_j A_k - \partial_k A_j) \partial^j \phi \\
 &= (-\Delta)^{-1} \nabla(\partial^k A_k)\cdot \nabla \phi + A_j \partial^j \phi 
 \end{align*}
 and by the definition of $u$
$$ (-\Delta)^{-1} \nabla \partial_t A_0 \cdot \nabla \phi = (-\Delta)^{-1} \nabla(\partial^j A_j -u)\cdot \nabla \phi \, , $$
so that we obtain
\begin{align*}
\mathcal{M}(\phi_+,\phi_-,A_+,A_-) & = 2i(-A_0 \partial_t \phi + (-\Delta)^{-1} \nabla u \cdot \nabla \phi + A_j \partial^j \phi) +  A_{\mu} A^{\mu} \phi\\
& = 2i(A_{\mu} \partial^{\mu} \phi + (-\Delta)^{-1} \nabla u \cdot \nabla \phi) + A_{\mu} A^{\mu} \phi \, ,
\end{align*}
which implies
\begin{align*}
&R(A,\phi) \\
& = Im(-\phi 2i(A_{\mu} \partial^{\mu} \overline{\phi} + (-\Delta)^{-1} \nabla u \cdot \nabla \overline{\phi}) + A_{\mu} \partial^{\mu}(|\phi|^2) +|\phi|^2 u + Im(\phi A_{\mu} A^{\mu} \overline{\phi}) \\
& = 2 Re(-\phi A_{\mu} \partial^{\mu} \overline{\phi}) -2 Re(\phi (-\Delta)^{-1} \nabla u \cdot \nabla \overline{\phi}) + 2 Re (A_{\mu} \phi \partial^{\mu} \overline{\phi}) + |\phi|^2 u \\
& = -2 Re(\phi \nabla \overline{\phi}) \cdot (-\Delta)^{-1} \nabla u + |\phi|^2 u \, ,
\end{align*}
so that
$$ (i\partial_t \pm |\nabla|)u_{\pm} = (\pm 2|\nabla|)^{-1}(-2 Re(\phi \nabla \overline{\phi}) \cdot (-\Delta)^{-1} \nabla u + |\phi|^2 u) \, , $$
and thus $u$ fulfills the linear equation
$$ \square u = -2 Re(\phi \nabla \overline{\phi}) \cdot (-\Delta)^{-1} \nabla u + |\phi|^2 u \, . $$
The data of $u$ fulfill by (\ref{10}) and (\ref{12}):
$$ u(0) = -\partial_t A_0(0) + \partial^j A_j(0) = - \dot{a}_{00} + \partial^j a_{0j} = 0 $$
and, using that $A$ is a solution of (\ref{16}) and also (\ref{15}):
\begin{align*}
\partial_t u(0) & = -\partial_t^2 A_0(0) + \partial_t \partial^j A_j(0) = -\partial^j \partial_j A_0(0) - j_{0_{|t=0}} + \partial_t \partial^j A_j(0) \\
& = -\partial^j (\partial_j A_0(0) - \partial_t A_j(0)) - j_{0_{|t=0}} = \partial^j F_{0j}(0) -  j_{0_{|t=0}} \\
& = Im(\phi_0 \overline{\phi}_1) - j_{0_{|t=0}} =  j_{0_{|t=0}}- j_{0_{|t=0}} = 0 \, .
\end{align*}
By uniqueness this implies $u = 0$. Thus the Lorenz condition $\partial^{\mu} A_{\mu} =0$ is satisfied. Under the Lorenz condition however we know that $\mathcal{M}$ in (\ref{3.1}) is the same as $M$ in (\ref{2.1}) and (\ref{17}).  Moreover by (\ref{3.1}) we obtain
\begin{align*}
(\square - m^2) \phi & = (i\partial_t - \langle \nabla \rangle_m)(i \partial_t + \langle \nabla \rangle_m) \phi_+ + (i\partial_t + \langle \nabla \rangle_m)(i \partial_t - \langle \nabla \rangle_m) \phi_- \\
& = (-(i \partial_t - \langle \nabla \rangle_m)  + (i \partial_t + \langle \nabla \rangle_m))(2 \langle \nabla \rangle_m)^{-1} \mathcal{M}(\phi_+,\phi_-,A_+,A_-) \\
& = \mathcal{M}(\phi_+,\phi_-,A_+,A_-) = M(\phi,A) \, .
\end{align*}
Thus $\phi$ satisfies (\ref{17}). Because (\ref{16}),(\ref{17}) is equivalent to (\ref{1}),(\ref{2}), where $F_{\mu \nu} := \partial_{\mu} A_{\nu} - \partial_{\nu} A_{\mu}$, we also have that $F_{k0}$ satisfies (\ref{2.10}) and $F_{kl}$ satisfies (\ref{2.11}).

What remains to be shown are the following properties of the electromagnetic field $F_{\mu \nu}$:
\begin{align*}
&F_{\mu \nu} \in X_+^{s-1,\frac{1}{2}+}[0,T] + X_-^{s-1,\frac{1}{2}+}[0,T] & \mbox{in the case} \, n=3\, , \\
&|\nabla|^{-\epsilon} F_{\mu \nu} \in X_+^{s-1+\epsilon,\frac{1}{2}+}[0,T] + X_-^{s-1+\epsilon,\frac{1}{2}+}[0,T] &\mbox{in the case} \, n=2 \, .
\end{align*}
Using well-known results for Bourgain type spaces these properties  in the case $n=3$ are reduced to:
\begin{align}
\label{4.10}
&F_{\mu \nu} (0) \in H^{s-1} \\
\label{4.11}
&\partial_t  F_{\mu \nu}(0) \in H^{s-2} \quad \mbox{and} \\
\label{4.14}
&\square F_{\mu \nu} \in X_+^{s-2,-\frac{1}{2}+}[0,T] + X_-^{s-2,-\frac{1}{2}+}[0,T] \, .
\end{align}
Similarly it suffices to show in the case $n=2$:
\begin{align}
\label{4.12}
&|\nabla|^{-\epsilon} F_{\mu \nu} (0) \in H^{s-1+\epsilon} \, ,
\\
\label{4.13}
&|\nabla|^{-\epsilon} \partial_t  F_{\mu \nu}(0) \in H^{s-2+\epsilon} \quad \mbox{and} \\
\label{4.15}
&|\nabla|^{-\epsilon} \square F_{\mu \nu} \in X_+^{s-2+\epsilon,-\frac{1}{2}+}[0,T] + X_-^{s-2+\epsilon,-\frac{1}{2}+}[0,T] \, ,
\end{align}
where in all cases $\square F_{\mu \nu} $ is given by (\ref{2.10}) and (\ref{2.11}).

The properties (\ref{4.10}) and (\ref{4.12}) are given by (\ref{8}). 

Next we consider the case $n=3$ and prove (\ref{4.11}). By (\ref{1}) we have
$$\partial_t F_{{0k}_{| t=0}} = - \partial_t F_{{k0}_{|t=0}} = - \partial^l F_{{kl}_{| t=0}} + j_{k_{| t=0}} \, . $$ 
By (\ref{8}) we have $\partial^l F_{{kl}_{| t=0}} \in H^{s-2}$, so it remains to prove $j_{k_{| t=0}} \in H^{s-2}$. We have
$$ j_{k_{| t=0}} = Im(\phi_0 \overline{\partial_k \phi_0}) + |\phi_0|^2 a_{0k} \, . $$
First we show for $s=\frac{3}{4}+$ and $r=\frac{1}{2}+$ :
$$\| |\phi_0|^2 a_{0k}\|_{H^{s-2}} \lesssim \|\phi_0\|^2_{H^s} \| \nabla a_{0k}\|_{H^{r-1}} < \infty \, . $$
We start with the estimate
$$ \| |\phi_0|^2 a_{0k}\|_{H^{-\frac{5}{4}+}} \lesssim \| |\phi_0|^2 \|_{H^{\frac{3}{4}+,\frac{6}{5}}} \| a_{0k} \|_{H^{-\frac{1}{2}+,6}} \, , $$
which follows by duality from the estimate
\begin{align*}
\| w |\phi_0|^2 \|_{H^{\frac{1}{2}-,\frac{6}{5}}}
 & \lesssim \|(\langle \nabla \rangle^{\frac{1}{2}-} w) |\phi_0|^2 \|_{L^{\frac{6}{5}}} + \| w \langle \nabla \rangle^{\frac{1}{2}-} (|\phi_0|^2) \|_{L^{\frac{6}{5}}} \\
& \lesssim \|\langle \nabla \rangle^{\frac{1}{2}-} w\|_{L^4} \| |\phi_0|^2 \|_{L^{\frac{12}{7}}} + \| w |_{L^{12-}} \|\langle \nabla \rangle^{\frac{1}{2}-} (|\phi_0|^2) \|_{L^{\frac{4}{3}+}} \\
& \lesssim \|w\|_{H^{\frac{5}{4}-}} \| |\phi_0|^2 \|_{H^{\frac{3}{4}+,\frac{6}{5}}} \, .
\end{align*}
Moreover
$$ \| |\phi_0|^2 \|_{H^{\frac{3}{4}+,\frac{6}{5}}} \lesssim \| \phi_0 \langle \nabla \rangle^{\frac{3}{4}+} \phi_0 \|_{L^{\frac{6}{5}}} \lesssim \|\phi_0\|_{L^3} \| \langle \nabla \rangle^{\frac{3}{4}+} \phi_0 \|_{L^2} \lesssim \|\phi_0\|_{H^{\frac{1}{2}}} \|\phi_0\|_{H^{\frac{3}{4}+}} \, ,  $$  
so that we arrive at
$$\| |\phi_0|^2 a_{0k}\|_{H^{s-2}} \lesssim  \|\phi_0\|^2_{H^s} \| a_{0k}\|_{H^{r-1,6}}      \lesssim \|\phi_0\|^2_{H^s} \| \nabla a_{0k}\|_{H^{r-1}} \, . $$
Next we show
$$ \| \phi_0 \nabla \phi_0 \|_{H^{s-2}} \lesssim \|\phi_0\|_{H^s} \|\nabla \phi_0\|_{H^{s-1}} < \infty \, . $$
which by duality follows from the estimate
\begin{align*}
\|\phi_0 w \|_{H^{1-s}} &=\|\phi_0 w \|_{H^{\frac{1}{4}-}} \lesssim \| \langle \nabla \rangle^{\frac{1}{4}-} \phi_0 \|_{L^3} \|w\|_{L^6} + \|\phi_0\|_{L^3} \| \langle \nabla \rangle^{\frac{1}{4}-} w \|_{L^6} \\
& \lesssim \|\phi_0 \|_{H^{\frac{3}{4}-}} \|w\|_{H^1} + \|\phi_0\|_{H^{\frac{1}{2}}} \| w \|_{H^{\frac{5}{4}-}} \\
& \lesssim \|\phi_0\|_{H^s} \|w\|_{H^{2-s}} \, .
\end{align*}
Altogether we have shown $\partial_t F_{{0k}_{| t=0}} \in H^{s-2}$. 

Moreover
$$\partial_t F_{jk} = \partial_0(\partial_j A_k - \partial_k A_j) = \partial_j(\partial_k A_0 + F_{0k}) - \partial_k(\partial_j A_0 + F_{0j}) = \partial_j F_{0k} - \partial_k F_{0j} \, ,$$
so that by (\ref{8}) we obtain
$$ \partial_t F_{{jk}_{| t=0}} = \partial_j F_{{0k}_{| t=0}} - \partial_k F_{{0j}_{| t=0}} \in H^{s-2} \, . $$

Next we consider the case $n=2$ and prove (\ref{4.13}). By (\ref{1}) we have
$$|\nabla|^{-\epsilon} \partial_t F_{{0k}_{| t=0}} = -|\nabla|^{-\epsilon} \partial_t F_{{k0}_{| t=0}} = -|\nabla|^{-\epsilon} \partial^l F_{{kl}_{| t=0}} +
|\nabla|^{-\epsilon} j_{{k}_{| t=0}} \, . $$
By (\ref{8}) we have $|\nabla|^{-\epsilon} \partial^l F_{{kl}_{| t=0}} \in H^{s-2+\epsilon}$, so that it remains to prove
$$|\nabla|^{-\epsilon} j_{{k}_{| t=0}} = |\nabla|^{-\epsilon} Im(\phi_0 \overline{\partial_k \phi_0}) + |\nabla|^{-\epsilon} (|\phi_0|^2 a_{0k}) \in H^{s-2+\epsilon} \, . $$
First we show for $s=\frac{3}{4}+\delta$ and $r=\frac{1}{4}+\delta$ the estimate
\begin{equation}
\label{4.16}
\||\nabla|^{-\epsilon} (|\phi_0|^2 a_{0k}) \|_{H^{s-2+\epsilon}} \lesssim \|\phi_0\|_{H^s}^2 \| |\nabla|^{1-\epsilon} a_{0k}\|_{H^{r-1+\epsilon}} < \infty \, .
\end{equation} 
We obtain
\begin{align*}
\||\nabla|^{-\epsilon} (|\phi_0|^2 a_{0k}) \|_{H^{s-2+\epsilon}} &\lesssim \| |\phi_0|^2 a_{0k} \|_{H^{-\frac{5}{4}+\epsilon+\delta,\frac{2}{1+\epsilon}}} \\ &\lesssim\| |\phi_0|^2\|_{H^{\frac{3}{4}-\epsilon-\delta,\frac{2}{2-\epsilon}}} \|a_{0k} \|_{H^{-\frac{3}{4}+\epsilon+\delta,\frac{2}{\epsilon}}} \\
&\lesssim\||\phi_0|^2\|_{H^{\frac{3}{4}-\epsilon-\delta,\frac{2}{2-\epsilon}}} \| |\nabla|^{1-\epsilon} a_{0k} \|_{H^{-\frac{3}{4}+\epsilon+\delta,2}} \\
& \lesssim \|\phi_0\|_{H^s}^2   \| |\nabla|^{1-\epsilon}a_{0k} \|_{H^{r-1+\epsilon}} \, .
\end{align*}
The first and third estimate follows from Sobolev's embedding, the last estimate follows from the crude estimate
$$\||\phi_0|^2\|_{H^{\frac{3}{4}-\epsilon-\delta,\frac{2}{2-\epsilon}}}  \lesssim \|\phi_0\|_{L^{\frac{2}{1-\epsilon}}} \| \langle \nabla \rangle^{\frac{3}{4}-\epsilon - \delta} \phi_0\|_{L^2} \lesssim \|\phi\|_{H^{\frac{3}{4}}}^2 $$
and the second estimate is by duality equivalent to
$$\| |\phi_0|^2 w \|_{H^{\frac{3}{4}-\epsilon-\delta,\frac{2}{2-\epsilon}}}  \lesssim\| |\phi_0|^2\|_{H^{\frac{3}{4}-\epsilon-\delta,\frac{2}{2-\epsilon}}} \|w\|_{H^{\frac{5}{4}-\epsilon-\delta,\frac{2}{1-\epsilon}}} \, $$
which can be shown as follows
\begin{align*}
\| |\phi_0|^2 w \|_{H^{\frac{3}{4}-\epsilon-\delta,\frac{2}{2-\epsilon}}} & 
\lesssim \| \langle \nabla \rangle^{\frac{3}{4}-\epsilon-\delta} 
(|\phi_0|^2) w \|_{L^{\frac{2}{2-\epsilon}}} + \| |\phi_0|^2 \langle \nabla \rangle^{\frac{3}{4}-\epsilon -\delta} w \|_{L^{\frac{2}{2-\epsilon}}} \\
& \lesssim \| |\phi_0|^2 \|_{H^{\frac{3}{4}-\epsilon -\delta,\frac{2}{2-\epsilon}}}
 \|w\|_{L^{\infty}} 
+ \| |\phi|^2 \|_{L^{\frac{4}{3}}} \|w\|_{H^{\frac{3}{4}-\epsilon-\delta,\frac{4}{1-2\epsilon}}} \\
& \lesssim \| |\phi_0|^2 \|_{H^{\frac{3}{4}-\epsilon -\delta,\frac{2}{2-\epsilon}}} \|w\|_{H^{\frac{5}{4}-\epsilon-\delta,\frac{2}{1-\epsilon}}} \, ,
\end{align*}
so that (\ref{4.16}) is shown. Next we estimate
$$ \| |\nabla|^{-\epsilon} (\phi_0 \nabla \overline{\phi}_0)\|_{H^{s-2+\epsilon}} \lesssim \|\phi_0 \nabla \overline{\phi}_0\|_{H^{-\frac{5}{4}+\epsilon+\delta,\frac{2}{1+\epsilon}}} \lesssim \|\phi_0\|_{H^{\frac{3}{4}}} \| \nabla \phi_0\|_{H^{-\frac{1}{4}}} \, . $$
The first estimate follows by Sobolev's embedding and the second estimate by duality from
$$\|\phi_0 w\|_{H^{\frac{1}{4}}} \lesssim \|\phi_0\|_{H^{\frac{1}{4},4}} \|w\|_{L^4} + \|\phi_0\|_{L^8} \|w\|_{H^{\frac{1}{4},\frac{8}{3}}} \lesssim \|\phi_0\|_{H^{\frac{3}{4}}} \|w\|_{H^{\frac{5}{4}-\epsilon-\delta,\frac{2}{1-\epsilon}}} \, , $$
which implies $|\nabla|^{-\epsilon} \partial_t F_{{0k}_{| t=0}} \in H^{s-2+\epsilon}$.

Moreover similarly as in three dimensions (\ref{8}) implies
$$|\nabla|^{-\epsilon} \partial_t F_{{jk}_{| t=0}} = |\nabla|^{-\epsilon} \partial_j F_{{0k}_{| t=0}} - |\nabla|^{-\epsilon} \partial_k F_{{0j}_{| t=0}} \in H^{s-2+\epsilon} \, , $$
which implies (\ref{4.13}).

Our next aim is to prove (\ref{4.14}). We start with the quadratic terms in (\ref{2.10}). Using (\ref{2.10a}) and (\ref{2.10b}) the required estimate in the case $n=3$ follows from the following claim, which holds in the case $n=2$ as well:\\
{\bf Claim 1:} For $n=2$ and $n=3$ the following estimate holds
$$ \| \mathcal{B}_{\pm_1,\pm_2} (\phi_{\pm_1},\phi_{\pm_2}) \|_{X_{\pm}^{s-2,-\frac{1}{2}++}} \lesssim \| \phi_{\pm_1}\|_{X^{s,\frac{1}{2}+}_{\pm_1}}  \|
\phi_{\pm_2}\|_{X^{s,\frac{1}{2}+}_{\pm_2}} \, . $$ 
By (\ref{2.15}) and (\ref{2.16}) this is equivalent to
\begin{align*}
I & := \left| \int \frac{b_{\pm_1,\pm_2} (\eta,\xi-\eta) \widehat{w}(\tau,\xi)}{\langle \xi \rangle^{2-s} \langle \tau \pm |\xi| \rangle^{\frac{1}{2}--}} \frac{\widehat{u}(\lambda,\eta)}{\langle \eta \rangle^s \langle \lambda \pm_1 |\eta| \rangle^{\frac{1}{2}+}} \frac{\widehat{v}(\tau - \lambda,\xi - \eta)}{\langle \xi - \eta \rangle^s \langle \tau - \lambda \pm_2 |\xi - \eta| \rangle^{\frac{1}{2}+}} \right| \\
& \lesssim \|u\|_{L^2} \|v\|_{L^2} \|w\|_{L^2} \, .
\end{align*}
By use of (\ref{2.17}) and (\ref{angle1}) we obtain
\begin{align*}
|b_{\pm_1,\pm_2}(\eta,\xi-\eta) &\lesssim m(|\eta|+|\xi-\eta|) +|\eta| |\xi - \eta| \angle(\pm_1 \eta,\pm_2 (\xi-\eta)) \\
& \lesssim m(|\eta| + |\xi - \eta|) + |\eta||\xi-\eta| \Big(\Big(\frac{\langle \tau \pm |\xi|\rangle}{\min(\langle \eta \rangle, \langle \xi - \eta \rangle)} \Big)^{\frac{1}{2}--} \\
&+ \frac{\langle \lambda \pm_1 |\eta| \rangle^{\frac{1}{2}} + \langle \tau - \lambda \pm_2 |\xi - \eta| \rangle^{\frac{1}{2}}}{\min(\langle \eta \rangle,\langle \xi-\eta \rangle)^{\frac{1}{2}}} \Big) \, .
\end{align*}
By symmetry we may assume $|\eta| \le |\xi - \eta|$. We first estimate the second summand.\\
{\bf Case 1:} $(\frac{\langle \lambda \pm_1 |\eta| \rangle}{|\eta|})^{\frac{1}{2}}$ dominant.\\
We obtain 
$$
I \lesssim  \int  \frac{\widehat{w}(\tau,\xi)}{\langle \xi \rangle^{2-s} \langle \tau \pm |\xi| \rangle^{\frac{1}{2}--}}
\frac{\widehat{u}(\lambda,\eta)}{\langle \eta \rangle^{s-\frac{1}{2}}} \frac{\widehat{v}(\tau - \lambda,\xi - \eta)}{\langle \xi - \eta \rangle^{s-1} \langle \tau - \lambda \pm_2 |\xi - \eta| \rangle^{\frac{1}{2}+}} \, . $$
An application of Theorem \ref{Theorem3} with $s_0=2-s$ , $b_0 = \frac{1}{2}--$ , $s_1=s-\frac{1}{2}$ , $b_1 =0$,  $s_2=s-1$ , $b_2= \frac{1}{2}+$ gives the required estimate for $ s > \frac{3}{4}$. \\ 
{\bf Case 2:} $(\frac{\langle \lambda - \tau \pm_2 |\xi - \eta| \rangle}{|\eta|})^{\frac{1}{2}}$ dominant.\\
We obtain 
$$
I \lesssim  \int  \frac{\widehat{w}(\tau,\xi)}{\langle \xi \rangle^{2-s} \langle \tau \pm |\xi| \rangle^{\frac{1}{2}--}}
\frac{\widehat{u}(\lambda,\eta)}{\langle \eta \rangle^{s-\frac{1}{2}}\langle \lambda \pm_1 |\eta| \rangle^{\frac{1}{2}+}} \frac{\widehat{v}(\tau - \lambda,\xi - \eta)}{\langle \xi - \eta \rangle^{s-1}} \, . $$
Theorem \ref{Theorem3} with $s_0=2-s$ , $b_0 = \frac{1}{2}--$ , $s_1=s-\frac{1}{2}$ , $b_1 =\frac{1}{2}+$ , $s_2=s-1$ , $b_2=0$ gives the claimed estimate. \\
{\bf Case 3:} $ (\frac{ \langle \tau \pm |\xi| \rangle}{|\eta|})^{\frac{1}{2}--}$ dominant. \\
{\bf Case 3a.} Either $n=3$ and $\pm_1$ and $\pm_2$ are the same signs or $n=2$.\\
We obtain 
$$
I \lesssim  \int  \frac{\widehat{w}(\tau,\xi)}{\langle \xi \rangle^{2-s}}
\frac{\widehat{u}(\lambda,\eta)}{\langle \eta \rangle^{s-\frac{1}{2}--}\langle \lambda \pm_1 |\eta| \rangle^{\frac{1}{2}+}} \frac{\widehat{v}(\tau - \lambda,\xi - \eta)}{\langle \xi - \eta \rangle^{s-1} \langle \tau - \lambda \pm_2 |\xi - \eta| \rangle^{\frac{1}{2}+}} \, . $$
Apply Theorem \ref{Theorem3} with $s_0=2-s$ , $b_0 = 0$ , $s_1=s-\frac{1}{2}--$ , $b_1 =\frac{1}{2}+$ , $s_2=s-1$,  $b_2= \frac{1}{2}+$. This is possible for $s>\frac{3}{4}$, where we remark that condition (\ref{2.13}) of Theorem \ref{Theorem3} has not to be fulfilled in the case $n=3$, because $\pm_1$ and $\pm_2$ are the same signs. (Otherwise (\ref{2.13}) would require $s > \frac{5}{6}$ for $n=3$.)\\
{\bf Case 3b:} $n=3$ and $\pm_1$ and $\pm_2$ are different signs.\\
In this case we use the improved estimate (\ref{angle2}) for the angle
\begin{align*}
\angle (\pm_1 \eta,\pm_2(\xi - \eta) 
& \lesssim \frac{|\xi|^{\frac{1}{2}}}{|\eta|^{\frac{1}{2}} |\eta - \xi|^{\frac{1}{2}}} \min(|\eta|,|\eta - \xi|)^{0++} \langle \tau \pm |\xi| \rangle^{\frac{1}{2}--} \\
&
\hspace{5em} + \frac{\langle \lambda \pm_1 |\eta| \rangle^{\frac{1}{2}} + \langle \tau - \lambda \pm_2 |\xi - \eta \rangle^{\frac{1}{2}}}{\min(|\eta|,|\eta - \xi|)^{\frac{1}{2}}} \, .
\end{align*}
It only remains to consider the first term on the right hand side because the other terms are the same as before. We obtain in this case
$$
I \lesssim  \int  \frac{\widehat{w}(\tau,\xi)}{\langle \xi \rangle^{\frac{3}{2}-s}}
\frac{\widehat{u}(\lambda,\eta)}{\langle \eta \rangle^{s-\frac{1}{2}--}\langle \lambda \pm_1 |\eta| \rangle^{\frac{1}{2}+}} \frac{\widehat{v}(\tau - \lambda,\xi - \eta)}{\langle \xi - \eta \rangle^{s-\frac{1}{2}} \langle \tau - \lambda \pm_2 |\xi - \eta| \rangle^{\frac{1}{2}+}}  $$
and apply Theorem \ref{Theorem3} with $s_0=\frac{3}{2}-s$ , $b_0 = 0$ , $s_1=s-\frac{1}{2}--$ , $b_1 =\frac{1}{2}+$ , $s_2=s-\frac{1}{2}$ , $b_2= \frac{1}{2}+$. This requires only $s > \frac{2}{3}$ for condition (\ref{2.13}). 

Finally we estimate the first summand $m(|\eta| + |\xi - \eta|)$. Assuming $|\xi - \eta| \le |\eta|$ without loss of generality we obtain the estimate
$$
I \lesssim  \int  \frac{\widehat{w}(\tau,\xi)}{\langle \xi \rangle^{2-s} \langle \tau \pm |\xi| \rangle^{\frac{1}{2}--}}
\frac{\widehat{u}(\lambda,\eta)}{\langle \eta \rangle^{s-1}\langle \lambda \pm_1 |\eta| \rangle^{\frac{1}{2}+}} \frac{\widehat{v}(\tau - \lambda,\xi - \eta)}{\langle \xi - \eta \rangle^s \langle \tau - \lambda \pm_2 |\xi - \eta| \rangle^{\frac{1}{2}+}} $$
and use Theorem \ref{Theorem3} with $s_0=2-s$ , $b_0 = \frac{1}{2}--$ , $s_1=s-1$ , $b_1 =\frac{1}{2}+$ , $s_2=s$,  $b_2= \frac{1}{2}+$, which requires only $ s > \frac{1}{2}$.

The quadratic term in (\ref{2.11}) can be handled in the same way, because $\mathcal{C}$ fulfills (\ref{2.14}), which is the same as (\ref{2.17}) without the linear term.

Before considering the cubic terms we treat the corresponding result in dimension $n=2$ needed for the proof of (\ref{4.15}).\\
{\bf Claim 2:} In the case $n=2$ we have
$$ \| |\nabla|^{-\epsilon} \mathcal{B}_{\pm_1,\pm_2} (\phi_{\pm_1},\phi_{\pm_2}) \|_{X^{s-2+\epsilon,-\frac{1}{2}++}_{\pm}} \lesssim \|\phi_{\pm_1}\|_{X^{s,\frac{1}{2}+}_{\pm_1}} \|\phi_{\pm_2}\|_{X^{s,\frac{1}{2}+}_{\pm_2}} \, . $$
If the frequencies $\xi$ of the product fulfills $|\xi| \ge 1$ we can use claim 1. Otherwise we have $\langle \eta \rangle \sim \langle \xi - \eta \rangle$. Arguing similarly as for claim 1 we have to show in case 1:
$$
\left| \int  \frac{\widehat{w}(\tau,\xi)}{|\xi|^{\epsilon} \langle \tau \pm |\xi| \rangle^{\frac{1}{2}--}}
\frac{\widehat{u}(\lambda,\eta)}{\langle \eta \rangle^{2s-\frac{3}{2}}} \frac{\widehat{v}(\tau - \lambda,\xi - \eta)}{ \langle \tau - \lambda \pm_2 |\xi - \eta| \rangle^{\frac{1}{2}+}}\right| \lesssim \|u\|_{L^2} \|v\|_{L^2}\|w\|_{L^2} \, .$$ 
Our assumption $s > \frac{3}{4}$ implies $2s-\frac{3}{2} > 0$, so that we only have to show
$$ \| |\nabla|^{-\epsilon}(uv)\|_{X^{0,-\frac{1}{2}++}_{\pm}} \lesssim \|u\|_{L^2} \|v\|_{X^{0,\frac{1}{2}+}_{\pm_2}} \, . $$ 
Using $|\xi| \le 1$ the left hand side is bounded by
\begin{align*}
\| |\nabla|^{-\epsilon}(uv)\|_{L^2_t H^{-2}_x} \lesssim \|uv\|_{L^2_t H^{-2,\frac{2}{1+\epsilon}}_x} \lesssim \|uv\|_{L^2_t L^1_x} &\lesssim \|u\|_{L^2_{xt}} \|v\|_{L^{\infty}_t L^2_x} \\
& \lesssim \|u\|_{L^2_{xt}} \|v\|_{X^{0,\frac{1}{2}+}_{\pm_2}} \, . 
\end{align*}
Similar arguments hold in the other cases. This completes the treatment of the quadratic terms.

Next we consider the cubic terms in (\ref{2.10}) and (\ref{2.11}).\\
{\bf Claim 3:} In the case $n=3$ we have
$$ \| \partial_l(A_{\nu \pm_3} \phi_{\pm_1} \overline{\phi}_{\pm_2}) \|_{X^{s-2,-\frac{1}{2}++}_{\pm}} \lesssim \|A_{\nu \pm_3} \|_{X^{r,1-\epsilon_0}_{\pm_3}} \|\phi_{\pm_1}\|_{X^{s,\frac{1}{2}+}_{\pm_1}}
\|\phi_{\pm_2}\|_{X^{s,\frac{1}{2}+}_{\pm_2}} \, . $$
This follows from the estimate
\begin{align*}
\| A_{\nu \pm_3} \phi_{\pm_1} \overline{\phi}_{\pm_2} \|_{X^{s-1,-\frac{1}{2}++}_{\pm}} &\lesssim \|A_{\nu \pm_3}\|_{X^{r,1-\epsilon_0}_{\pm_3}} \| \phi_{\pm_1} \phi_{\pm_2}\|_{X^{\frac{1}{4}+,0}_{\pm}} \\
&\lesssim \|A_{\nu \pm_3} \|_{X^{r,1-\epsilon_0}_{\pm_3}} \|\phi_{\pm_1}\|_{X^{s,\frac{1}{2}+}_{\pm_1}}
\|\phi_{\pm_2}\|_{X^{s,\frac{1}{2}+}_{\pm_2}} \, ,
\end{align*}
which is obtained by Theorem \ref{Theorem3} with $s_0=1-s$ , $b_0=\frac{1}{2}--$ , $s_1 =r$ , $b_1 = 1-\epsilon_0$,  $s_2 = \frac{1}{4}+$ , $b_2=0$ for the first estimate and $s_0=-\frac{1}{4}-$ , $b_0=0$ , $s_1 =s_2 =s$ , $b_1 = b_2 = \frac{1}{2}+$ for the last step.\\
{\bf Claim 4:} In the case $n=2$ we have
\begin{align*}
 &\| |\nabla|^{-\epsilon} \partial_l(A_{\nu \pm_3} \phi_{\pm_1} \overline{\phi}_{\pm_2}) \|_{X^{s-2+\epsilon,-\frac{1}{2}++}_{\pm}} \\
 & \hspace{5em} \lesssim \||\nabla|^{1-\epsilon} A_{\nu \pm_3} \|_{X^{r-1+\epsilon,1-\epsilon_0}_{\pm_3}} \|\phi_{\pm_1}\|_{X^{s,\frac{1}{2}+}_{\pm_1}}
\|\phi_{\pm_2}\|_{X^{s,\frac{1}{2}+}_{\pm_2}} \, . 
\end{align*}
We first assume frequencies of $A_{\nu \pm_3}$ which are $\ge 1$. We have
\begin{align*}
&\| |\nabla|^{-\epsilon} \partial_l(A_{\nu \pm_3} \phi_{\pm_1} \overline{\phi}_{\pm_2}) \|_{X^{s-2+\epsilon,-\frac{1}{2}++}_{\pm}} \lesssim \| A_{\nu \pm_3} \phi_{\pm_1} \overline{\phi}_{\pm_2} \|_{X^{s-1,-\frac{1}{2}++}_{\pm}} \\
& \lesssim \| A_{\nu \pm_3}\|_{X^{r,1-\epsilon_0}_{\pm_3}} \| \phi_{\pm_1} \phi_{\pm_2} \|_{X^{\frac{1}{4}+,0}_{\pm}} \lesssim \| |\nabla|^{1-\epsilon} A_{\nu \pm_3}\|_{X^{r-1+\epsilon,1-\epsilon_0}_{\pm_3}} \| \phi_{\pm_1} \phi_{\pm_2} \|_{X^{\frac{1}{4}+,0}_{\pm}} \,
\end{align*}
where the second estimate follows from Theorem \ref{Theorem3} with $s_0=1-s$ , $b_0=\frac{1}{2}--$,  $s_1 =r$ , $b_1 = 1-\epsilon_0$ , $s_2 = \frac{1}{4}+$ , $b_2=0$. Now the estimate
\begin{align*}
\|\phi_{\pm_1}\phi_{\pm_2} \|_{X^{\frac{1}{4}+,0}} & \lesssim 
\|\phi_{\pm_1}\|_{L^4_t H^{\frac{1}{4}+,4}_x} \|\phi_{\pm_2}\|_{L^4_t L^4_x} + \|\phi_{\pm_1}\|_{L^4_t L^4_x} \|\phi_{\pm_2}\|_{L^4_t H^{\frac{1}{4}+,4}_x} \\
& \lesssim \|\phi_{\pm_1}\|_{L^4_t H^{\frac{3}{4}+}_x} 
\|\phi_{\pm_2}\|_{L^4_t H^{\frac{1}{2}}_x} + \|\phi_{\pm_1}\|_{L^4_t H^{\frac{1}{2}}_x} \|\phi_{\pm_2}\|_{L^4_t H^{\frac{3}{4}+}_x} \\
& \lesssim \|\phi_{\pm_1}\|_{X^{\frac{3}{4}+,\frac{1}{2}+}_{\pm_1}} \|\phi_{\pm_2}\|_{X^{\frac{3}{4}+,\frac{1}{2}+}_{\pm_2}}
\end{align*}
gives the desired bound. For low frequencies $\le 1$ of $A_{\nu \pm_3}$ we obtain
\begin{align*}
&\| A_{\nu \pm_3} \phi_{\pm_1} \overline{\phi}_{\pm_2}\|_{X^{s-1,-\frac{1}{2}++}_{\pm}} \lesssim \| A_{\nu \pm_3} \phi_{\pm_1} \phi_{\pm_2}\|_{L^2_t L^2_x} \\
&\lesssim \| A_{\nu \pm_3}\|_{L^2_t L^{\frac{2}{\epsilon}}_x} \| \phi_{\pm_1} \|_{L^{\infty}_t L^{\frac{4}{1-\epsilon}}_x} \| \phi_{\pm_2}\|_{L^{\infty}_t L^{\frac{4}{1-\epsilon}}_x} \\
&\lesssim \| |\nabla|^{1-\epsilon} A_{\nu \pm_3}\|_{L^2_t L^2_x} \| \phi_{\pm_1} \|_{L^{\infty}_t H^{\frac{1+\epsilon}{2}}_x} \|\phi_{\pm_2}\|_{L^{\infty}_t H^{\frac{1+\epsilon}{2}}_x} \\
&\lesssim \| |\nabla|^{1-\epsilon} A_{\nu \pm_3}\|_{X^{r-1+\epsilon,1-\epsilon_0}_{\pm_3}} \| \phi_{\pm_1} \|_{X^{s,\frac{1}{2}+}_{\pm_1}} \|\phi_{\pm_2}\|_{X^{s,\frac{1}{2}+}_{\pm_2}} \, ,
\end{align*}
where we used in the last step that we have low frequencies of $A_{\nu \pm_3}$.

Finally we have to consider the term $\partial_t(A_k |\phi|^2)$. Using $\partial_t \phi = i \langle \nabla \rangle_m (\phi_+ - \phi_-)$ and $\partial_t A_k = i |\nabla|(A_{k_+} -A_{k_-})$ we obtain
 $$\partial_t (A_k |\phi|^2) = i|\phi|^2 |\nabla|(A_{k_+}-A_{k_-}) + iA_k \langle \nabla \rangle_m (\phi_+ - \phi_-) \overline{\phi} + A_k \phi \overline{i \langle \nabla \rangle_m (\phi_+ - \phi_-)} \, . $$
Now $|\nabla|(A_k |\phi|^2) \sim (|\nabla| A_k) |\phi|^2 + A_k \phi |\nabla|\phi$, so that one has to consider terms of the type $|\nabla|(A_k |\phi|^2)$ and $A_k \phi \langle \nabla \rangle \phi$. The first term was considered in claim 3. Thus it remains to show claim 5 and claim 6.\\
{\bf Claim 5:} In the case $n=3$ and $n=2$ we have
$$ \| A_{k \pm_3} \phi_{\pm_1} \langle \nabla \rangle \phi_{\pm_2} \|_{X^{s-2,-\frac{1}{2}++}_{\pm}} \lesssim \|A_{k \pm_3} \|_{X^{r,1-\epsilon_0}_{\pm_3}} \|\phi_{\pm_1}\|_{X^{s,\frac{1}{2}+}_{\pm_1}}
\|\phi_{\pm_2}\|_{X^{s,\frac{1}{2}+}_{\pm_2}} \, . $$
We obtain
\begin{align*}
\| A_{k \pm_3} \phi_{\pm_1} \langle \nabla \rangle \phi_{\pm_2} \|_{X^{s-2,-\frac{1}{2}++}_{\pm}} & \lesssim \|A_{k \pm_3} \|_{X^{r,1-\epsilon_0}_{\pm_3}} \|\phi_{\pm_1}
\langle \nabla \rangle \phi_{\pm_2}\|_{X^{m,0}_{\pm}} \\
& \lesssim \|A_{k \pm_3} \|_{X^{r,1-\epsilon_0}_{\pm_3}} \|\phi_{\pm_1}\|_{X^{s,\frac{1}{2}+}_{\pm_1}}
\|\langle \nabla \rangle \phi_{\pm_2}\|_{X^{s-1,\frac{1}{2}+}_{\pm_2}} \, ,
\end{align*} 
where $m=-\frac{1}{2}+$ for $n=3$ and $m=-\frac{1}{4}+$ for $n=2$. For the first step we used Theorem \ref{Theorem3} with $s_0=2-s$ , $b_0 = \frac{1}{2}--$ , $s_1=r$ , $b_1=1-\epsilon_0$, $s_2=m$ , $b_2=0$, and for the second step with $s_0=-m$ , $b_0 = 0$ , $s_1=s$ , $b_1=\frac{1}{2}+$, $s_2=s-1$ , $b_2=\frac{1}{2}+$.

Finally we have to consider the case of low frequencies of $A_{k \pm_3}$ and / or the product in the case $n=2$.\\
{\bf Claim 6:} In the case $n=2$ we have
\begin{align*}
 &\||\nabla|^{-\epsilon} (A_{k \pm_3} \phi_{\pm_1} \langle \nabla \rangle \phi_{\pm_2}) \|_{X^{s-2+\epsilon,-\frac{1}{2}++}_{\pm}} \\ 
 &\lesssim \||\nabla|^{1-\epsilon} A_{k \pm_3} \|_{X^{r-1+\epsilon,1-\epsilon_0}_{\pm_3}} \|\phi_{\pm_1}\|_{X^{s,\frac{1}{2}+}_{\pm_1}}
\|\phi_{\pm_2}\|_{X^{s,\frac{1}{2}+}_{\pm_2}} \, . 
\end{align*}
For low frequencies of $A_{k \pm_3}$ we want to show first
\begin{align}
\nonumber
 &\||\nabla|^{-\epsilon} (A_{k \pm_3} \phi_{\pm_1} \langle \nabla \rangle \phi_{\pm_2}) \|_{X^{s-2+\epsilon,-\frac{1}{2}++}_{\pm}} \\
 \label{*}
 &\lesssim \||\nabla|^{1-\epsilon} A_{k \pm_3} \|_{X^{r-1+\epsilon,1-\epsilon_0}_{\pm_3}} \|\phi_{\pm_1}
\nabla \phi_{\pm_2}\|_{X^{-\frac{1}{4},0}_{\pm}} \, ,
\end{align}
which follows from
\begin{align*}
&\int  \frac{\widehat{w}(\tau,\xi)}{\langle \xi \rangle^{2-s-\epsilon}|\xi|^{\epsilon} \langle \tau \pm |\xi| \rangle^{\frac{1}{2}--}}
\frac{\widehat{u}(\lambda,\eta)}{\langle \eta \rangle^{r-1+\epsilon} |\eta|^{1-\epsilon} \langle \lambda \pm_3 |\eta| \rangle^{1-\epsilon_0}} \frac{\widehat{v}(\tau - \lambda,\xi - \eta)}{ \langle \xi - \eta \rangle^{-\frac{1}{4}}} \\
&\lesssim \|u\|_{L^2} \|v\|_{L^2}\|w\|_{L^2} \, .
\end{align*} 
Now $|\eta| \le 1$ implies $\langle \eta \rangle \sim 1$ and $\langle \xi \rangle \sim \langle \xi - \eta \rangle $, so that the left hand side is equivalent to
$$\int  \frac{\widehat{w}(\tau,\xi)}{\langle \xi \rangle^{\frac{7}{4}-s-\epsilon}|\xi|^{\epsilon} \langle \tau \pm |\xi| \rangle^{\frac{1}{2}--}}
\frac{\widehat{u}(\lambda,\eta)}{ |\eta|^{1-\epsilon} \langle \lambda \pm_3 |\eta| \rangle^{1-\epsilon_0}} \widehat{v}(\tau - \lambda,\xi - \eta) \, .$$
This implies that (\ref{*}) is equivalent to
$$\|fg\|_{L^2_{xt}} \lesssim \| |\nabla|^{\epsilon} f \|_{X^{\frac{7}{4}-s-\epsilon,\frac{1}{2}--}_{\pm}} \| |\nabla|^{1-\epsilon} g \|_{X^{0,1-\epsilon_0}_{\pm_3}} \, . $$
This is true, because
\begin{align*}
\| fg \|_{L^2_{xt}} \lesssim \|f\|_{L^4_t L^{\frac{2}{1-\epsilon}}_x} \|g\|_{L^4_t L^{\frac{2}{\epsilon}}_x} & \lesssim \| |\nabla|^{\epsilon} f \|_{L^4_t L^2_x} \| |\nabla|^{1-\epsilon} g \|_{L_t^4 L^2_x} \\
& \lesssim \| |\nabla|^{\epsilon} f \|_{X^{\frac{7}{4}-s-\epsilon,\frac{1}{2}--}_{\pm}} \| |\nabla|^{1-\epsilon} g \|_{X^{0,1-\epsilon_0}_{\pm_3}} \, ,
\end{align*}
so that (\ref{*}) holds. Combining this with the estimate
$$ \|\phi_{\pm_1} \langle \nabla \rangle \phi_{\pm_2} \|_{X^{-\frac{1}{4},0}_{\pm}} \lesssim \|\phi_{\pm_1}\|_{X^{s,\frac{1}{2}+}_{\pm_1}}  \|\phi_{\pm_2}\|_{X^{s,\frac{1}{2}+}_{\pm_2}} $$
proven for claim 5 we arrive at the claimed estimate.

From now on we may assume high frequencies of $A_{k \pm_3}$ and low frequencies of the product. This implies that the frequencies of $A_{k \pm_3}$ and $\phi_{\pm_1} \langle \nabla \rangle \phi_{\pm_2}$ are equivalent. First we want to show
\begin{equation}
\label{5.1a}
\||\nabla|^{-\epsilon} (A_{k \pm_3} \phi_{\pm_1} \langle \nabla \rangle \phi_{\pm_2}) \|_{X^{s-2-\epsilon,-\frac{1}{2}++}_{\pm}} \lesssim \|A_{k \pm_3}\|_{X^{r,1-\epsilon_0}_{\pm_3}} \|\phi_{\pm_1} \langle \nabla \rangle \phi_{\pm_2}\|_{X^{-\frac{1}{4},0}_{\pm}} \, ,
\end{equation}
which follows from
$$
\int  \frac{\widehat{w}(\tau,\xi)}{|\xi|^{\epsilon} \langle \tau \pm |\xi| \rangle^{\frac{1}{2}--}}
\frac{\widehat{u}(\lambda,\eta)}{\langle \eta \rangle^r \langle \lambda \pm_3 |\eta| \rangle^{1-\epsilon_0}} \frac{\widehat{v}(\tau - \lambda,\xi - \eta)}{ \langle \xi - \eta \rangle^{-\frac{1}{4}}} \lesssim \|u\|_{L^2} \|v\|_{L^2}\|w\|_{L^2} \, .$$ 
Using $\langle \eta \rangle \sim \langle \xi - \eta \rangle$ and $r > \frac{1}{4}$ this follows from
$$
\||\nabla|^{-\epsilon} (A_{k \pm_3} \phi_{\pm_1} \langle \nabla \rangle \phi_{\pm_2}) \|_{X^{0,-\frac{1}{2}++}_{\pm}} \lesssim \|A_{k \pm_3}\|_{X^{0,1-\epsilon_0}_{\pm_3}} \|\phi_{\pm_1} \langle \nabla \rangle \phi_{\pm_2}\|_{X^{0,0}_{\pm}} \, ,
$$
Using that the product has low frequencies the left hand side is bounded by
\begin{align*}
&\||\nabla|^{-\epsilon} (A_{k \pm_3} \phi_{\pm_1} \langle \nabla \rangle \phi_{\pm_2}) \|_{L^2_t H^{-2}_x} \lesssim \|A_{k \pm_3} \phi_{\pm_1} \langle \nabla \rangle \phi_{\pm_2} \|_{L^2_t H^{-2,\frac{2}{1+\epsilon}}_x} \\
& \lesssim \||A_{k \pm_3} \phi_{\pm_1} \langle \nabla \rangle \phi_{\pm_2}) \|_{L^2_t L^1_x} \lesssim \|A_{k \pm_3}\|_{L^{\infty}_t L^2_x} \|\phi_{\pm_1} \langle \nabla \rangle \phi_{\pm_2}\|_{L^2_t L^2_x} \\
&\lesssim
\|A_{k \pm_3}\|_{X^{0,1-\epsilon_0}_{\pm_3}} \|\phi_{\pm_1} \langle \nabla \rangle \phi_{\pm_2}\|_{X^{0,0}_{\pm}} \, ,
\end{align*}
so that we have proven (\ref{5.1a}). Combining this with the estimate
$$ \|\phi_{\pm_1} \langle \nabla \rangle \phi_{\pm_2} \|_{X^{-\frac{1}{4},0}_{\pm}} \lesssim \|\phi_{\pm_1}\|_{X^{s,\frac{1}{2}+}_{\pm_1}}  \|\phi_{\pm_2}\|_{X^{s,\frac{1}{2}+}_{\pm_2}} $$
proven for claim 5 we finally obtain claim 6 for low frequencies of the product and high frequencies of $A_{k \pm_3}$.

This completes the proof of (\ref{4.14}) and (\ref{4.15}) and also the proof of Theorem \ref{Theorem1}.
\end{proof}

\end{document}